\documentclass[10pt,reqno]{amsart}
\usepackage{amssymb,mathrsfs,graphicx,extpfeil,amsmath}
\usepackage{ifthen,latexsym,float,colortbl}
\usepackage{hyperref}
\usepackage[margin=1in]{geometry}
\usepackage{caption}
\usepackage{sidecap}
\usepackage{rotating,dsfont,stackengine}
\usepackage{amsthm}

\usepackage{colortbl}
\definecolor{black}{rgb}{0.0, 0.0, 0.0}
\definecolor{red}{rgb}{1.0, 0.5, 0.5}
\provideboolean{shownotes} 
\setboolean{shownotes}{true} 
\newcommand{\margnote}[1]{
	\ifthenelse{\boolean{shownotes}}%
	{\marginpar{\raggedright\tiny\texttt{#1}}}%
	{}%
}
\newcommand{\hole}[1]{
	\ifthenelse{\boolean{shownotes}}%
	{\begin{center} \fbox{ \rule {.25cm}{0cm} \rule[-.1cm]{0cm}{.4cm}
				\parbox{.85\textwidth}{\begin{center} \texttt{#1}\end{center}} \rule
				{.25cm}{0cm}}\end{center}} {} }

\graphicspath{{pics/}}

\title[Damped Euler system with attractive Riesz interaction forces]{Damped Euler system with attractive Riesz interaction forces}

\author[Choi]{Young-Pil Choi}
\address[Young-Pil Choi]{\newline Department of Mathematics \newline
	Yonsei University, 50 Yonsei-Ro, Seodaemun-Gu, Seoul 03722, Republic of Korea}
\email{ypchoi@yonsei.ac.kr}

\author[Jung]{Jinwook Jung}
\address[Jinwook Jung]{\newline Department of Mathematics and Institute of Pure and Applied Mathematics \newline Jeonbuk National University, 567 Baekje-daero, Deokjin-gu, Jeonju-si, Jeollabuk-do 54896,  Republic of Korea}
\email{2jwook12@gmail.com}

\author[Lee]{Yoonjung Lee}
\address[Yoonjung Lee]{\newline Department of Mathematics \newline
	Yonsei University, 50 Yonsei-Ro, Seodaemun-Gu, Seoul 03722, Republic of Korea}
\email{yjglee@yonsei.ac.kr}

\numberwithin{equation}{section}

\newtheorem{theorem}{Theorem}[section]
\newtheorem{lemma}{Lemma}[section]

\newtheorem{proposition}{Proposition}[section]
\newtheorem{remark}{Remark}[section]

\newcommand{\R}{\mathbb R}

\newcommand{\N}{\mathbb N}
\newcommand{\Z}{\mathbb Z}

\newcommand{\ls}{\lesssim}
\newcommand{\gs}{\gtrsim}
\newcommand{\T}{\mathbb T}

\newcommand{\mc}{\mathcal C}

\newcommand{\sfI}{\mathsf{I}}
\newcommand{\sfJ}{\mathsf{J}}
\newcommand{\sfK}{\mathsf{K}}
\newcommand{\sfL}{\mathsf{L}}

\newcommand{\bq}{\begin{equation}}
	\newcommand{\eq}{\end{equation}}
\newcommand{\e}{\varepsilon}
\newcommand{\lt}{\left}
\newcommand{\rt}{\right}

\newcommand{\pa}{\partial}

\makeatletter
\def\moverlay{\mathpalette\mov@rlay}
\def\mov@rlay#1#2{\leavevmode\vtop{%
		\baselineskip\z@skip \lineskiplimit-\maxdimen
		\ialign{\hfil$\m@th#1##$\hfil\cr#2\crcr}}}
\newcommand{\charfusion}[3][\mathord]{
	#1{\ifx#1\mathop\vphantom{#2}\fi
		\mathpalette\mov@rlay{#2\cr#3}
	}
	\ifx#1\mathop\expandafter\displaylimits\fi}
\makeatother

\begin{document}
\allowdisplaybreaks




\begin{abstract} We consider the barotropic Euler equations with pairwise attractive Riesz interactions and linear velocity damping in the periodic domain. We establish the global-in-time well-posedness theory for the system near an equilibrium state. We also analyze the large-time behavior of solutions showing the exponential rate of convergence toward the equilibrium state as time goes to infinity. 
\end{abstract}
\maketitle 


%
%
%
%
\section{Introduction}
In this paper, we analyze the barotropic Euler equations with pairwise attractive Riesz interactions and linear damping in velocity. More precisely, let $\rho=\rho(x,t)$ and $u= u(x,t)$ be the density and velocity of the fluid at time $t \in \R_+$ and position $x \in \T^d$ with $d\geq 1$, respectively. Then our main system reads as
\bq\label{np_ER}
\begin{aligned}
&\pa_t \rho + \nabla \cdot (\rho u) = 0, \quad x \in \T^d, \quad t > 0,\cr
&\pa_t (\rho u) + \nabla \cdot (\rho u \otimes u) + \nabla p= -\nu \rho u + \lambda \rho \nabla \Lambda^{\alpha-d} (\rho-c),\cr
&p = p(\rho) = \rho^\gamma, \quad \gamma \geq 1,
\end{aligned}
\eq
supplemented with the initial data
\bq\label{ini_ER}
(\rho,u)|_{t = 0} =: (\rho_0, u_0), \quad x \in \T^d.
\eq
The right-hand side of the momentum equations consists of the linear damping in velocity with $\nu > 0$ and the Riesz interactions, where the Riesz operator $\Lambda^{s}$ is defined by $(-\Delta)^{s/2}$. For $\alpha$, we deal with the super Coulombic case: 
\[
\max\{d-2, 0\}\leq \alpha <d.
\]
In this case, we observe
\[
\Lambda^{\alpha-d} (\rho-c) = c_{\alpha, d} K_\alpha * (\rho-c), \quad \mbox{with} \quad K_\alpha(x) = \frac{|x|^{-\alpha}}{\alpha}
\]
for some $c_{\alpha, d} > 0$. Here we use the convention $K_\alpha(x) = \log |x|$ for $\alpha = 0$. Note that the case $\alpha = d-2$ corresponds to the Coulombic interactions when $d \geq 3$, and the potential $K$ is called a pure Manev interaction potential when $\alpha = d-1$, see \cite{BDIV97, CJepre, Ill00, IVDB98}.  

The Riesz interaction force is attractive if $\lambda> 0$ and repulsive if $\lambda < 0$. 
In the current work, we consider the case of attractive interactions, i.e., $\lambda > 0$. The background state $c$ is given as a fixed positive constant, and we impose the neutrality condition: 
\bq\label{neu_con}
\int_{\T^d} (\rho - c) \,dx = 0.
\eq 
Since the total mass is conserved in time, without loss of generality, we may assume that $\rho$ is a probability density function, i.e., $\|\rho(t)\|_{L^1} = 1$ for all $t \geq 0$, and subsequently, we set $c=1$.

In the present work, we study the global Cauchy problem for the system \eqref{np_ER}. To be more specific, we are concerned with the global-in-time existence and uniqueness of regular solutions to the system \eqref{np_ER} and its large-time behavior under suitable assumptions on the initial data. To state our main results and compare them with existing results, inspired by \cite{BDIV97, CJepre, Ill00, IVDB98}, we consider two cases\footnote{Precisely, as mentioned, the pure Manev potential corresponds to the case $\alpha = d-1$. However, for simplicity of presentation, we include that critical case in the sub-Manev potential one. }:
\begin{itemize}
\item[(i)] $\max\{d-2, 0\}\leq \alpha \leq d-1$ (sub-Manev potential case) and
\item[(ii)] $d-1 < \alpha <d$ (super-Manev potential case).
\end{itemize}
The local-in-time well-posedness theory for the system \eqref{np_ER} with the zero background state is established in \cite{CJe22}. We would like to remark that the linear analysis presented in \cite{CJe22} illuminates that the system \eqref{np_ER} with $\lambda > 0$ is ill-posed in the absence of pressure. More recently, in the sub-Manev potential and zero background case, the global-in-time existence and uniqueness of regular solution for the system \eqref{np_ER} with $\nu = 0$ are constructed in \cite{DD22} for well-prepared initial data. In \cite{CJ23}, the global-in-time well-posedness and the large-time behavior of solutions to the system \eqref{np_ER} are obtained in the case of pressureless ($p \equiv 0$) and repulsive interaction $(\lambda < 0)$. For the Coulombic interaction case, the global existence theory for the system \eqref{np_ER} in two dimensions with $\nu=0$ and $\lambda < 0$ is established in \cite{IP13, LW14}. We refer to \cite{CJe22, CJ23, DD22} and references therein for the recent works on the system \eqref{np_ER} and related results. Note that 
%
%
the system \eqref{np_ER} is also related to the following fractional porous medium flow \cite{CSV13}:
\bq\label{fpms}
\pa_t \rho + \lambda \nabla \cdot (\rho \nabla \Lambda^{\alpha - d}(\rho - c)) = \Delta \rho^\gamma.
\eq
To be more specific, by considering the overdamped regime, i.e., $\nu = \frac1\e$ after a proper scaling of time $\pa_t \mapsto \e \pa_t$, we find from \eqref{np_ER} that 
\begin{align*}
\begin{aligned}
&\pa_t \rho^{(\e)} + \frac1\e \nabla \cdot (\rho^{(\e)} u^{(\e)}) = 0, \cr
&\pa_t (\rho^{(\e)} u^{(\e)}) + \frac1\e \nabla \cdot (\rho^{(\e)} u^{(\e)} \otimes u^{(\e)}) + \frac1\e \nabla p(\rho^{(\e)})= - \frac1{\e^2}\rho^{(\e)} u^{(\e)} + \frac\lambda\e \rho^{(\e)} \nabla \Lambda^{\alpha-d} (\rho^{(\e)}-c).
\end{aligned}
\end{align*}
Formally, if $\e \ll 1$, we deduce
\[
\frac1{\e}\rho^{(\e)} u^{(\e)} \sim  \lambda  \rho^{(\e)} \nabla \Lambda^{\alpha-d} (\rho^{(\e)}-c) - \nabla p(\rho^{(\e)}),
\]
and thus, if $\rho^{(\e)} \to \rho$ as $\e \to 0$ in some sense, we obtain that $\rho$ satisfies \eqref{fpms}. We refer to \cite{CJe21, LT17} for the study of the relaxation limit of damped Euler equations with nonlocal interaction forces.

%
%
%
%
%
%
\subsection{Reformulations and notations} In order to present our main results, we first reduce the system \eqref{np_ER} to a symmetric system. For the case $\gamma>1$, we consider the sound speed of $\rho$:
\bq\nonumber
\begin{aligned}
	s(\rho) &=\sqrt{p'(\rho)}.
\end{aligned}
\eq
We let
\bq\nonumber
h=N (s(\rho)-s(1)) 
\quad \text{and} \quad N:=\frac{2}{\gamma-1}.
\eq
Setting $\sigma:=\gamma^{1/2}N$, we can write $ h= \sigma (\rho^{1/N}-1)$, and the system \eqref{np_ER} with initial data \eqref{ini_ER} can be rewritten as a symmetric form:
\bq\label{np_ER2}
\begin{aligned}
&\pa_t h + u \cdot \nabla h +\frac{1}{N} (h+\sigma)(\nabla \cdot u) = 0, \quad x \in \T^d, \quad t > 0,\cr
&\pa_t u + u\cdot \nabla u +  \frac{1}{N}(h+ \sigma) \nabla h  = - \nu u + \sigma^{-N}\lambda \nabla \Lambda^{\alpha-d} \lt\{ (h+ \sigma )^{N} - \sigma^N  \rt\}
\end{aligned}
\eq
with 
initial data
\bq\label{ini_ER2}
(h,u)|_{t=0} =: (h_0:=\sigma (\rho_0^{1/N}-1),\, u_0), \quad x \in \T^d.
\eq
Note that the neutrality condition \eqref{neu_con} becomes
\bq\label{neu_con2}
\int_{\T^d} (h+\sigma)^N \,dx = \sigma^N.
\eq
On the other hand, when $\gamma=1$, we let 
$h := \ln \rho$ so that 
\eqref{np_ER} and \eqref{ini_ER} are transformed into the following system:
\bq\nonumber
\begin{aligned}
	&\pa_t h + u \cdot \nabla h +\nabla \cdot u = 0, \quad x \in \T^d, \quad t > 0,\cr
	&\pa_t u + u\cdot \nabla u +  \nabla h  = - \nu u + \lambda \nabla \Lambda^{\alpha-d}(e^h -1),
\end{aligned}
\eq
with initial data 
\bq\nonumber
(h,u)|_{t=0} =: (h_0:= \ln \rho_0,\, u_0), \quad x \in \T^d.
\eq
Similarly, in this case, the neutrality condition becomes 
\[
\int_{\T^d} e^h dx = 1.
\]

%
%
%
%
%
%
Here we introduce our simplified notations. Throughout the paper, we write $\sigma=\gamma^{1/2}N$ and $N=2/(\gamma-1)$. We frequently denote $\| \cdot \|_{X(\T^d)}$ by $\| \cdot \|_{X}$, and $\int_{\T^d}$ by $\int$ for simplicity.
We also denote $\|(h, u)(t)\|_{X}^2:=\|h(t)\|_{X}^2 +\|u(t)\|_{X}^2$, 
We use $C$ as a generic constant which does not depend on $\lambda$, $\nu$, and $t$, 
and denote by $C_{\alpha, \beta, \cdots}$ or $C(\alpha,\beta, \cdots )$ a generic constant depending on $\alpha, \beta, ...$.
We use the notation $A(t) \sim B(t)$ in the sense that there exists a positive $C$ independent of $t$ such that $C^{-1} A(t) \leq B(t) \leq C A(t)$, 
the notation $A \sim_{\alpha} B$ in the sense
such that $C_1 A \leq B \leq C_2 A$ for some $C_1, C_2 >0$ depending only on $\alpha$.

\begin{remark}\label{rem1}
	For $m\in \mathbb{N}$, the classical Sobolev spaces is defined under the norm 
	\bq\nonumber
	\|f\|_{H^m}^2 := \sum_{0\leq k \leq m}\|\pa^{\alpha} f\|_{L^2}^2 
	\eq 
	for multi-index $\alpha$ of order $|\alpha|=k$. 
	Here, one observes
	\bq\nonumber
	\|f\|_{H^m }^2 \sim_{d, m} \sum_{0\leq k\leq m}\|\Lambda^k f\|_{L^2 }^2 \quad \text{for} \quad  m\ge 0,
	\eq 
	since
	\bq\label{equiv_form}
	\sum_{|\alpha|=k}\|\pa^{\alpha} f\|_{L^2 }^2
	\sim_{d, k} \|\Lambda^k f\|_{L^2 }^2 \quad \text{for} \quad  k\ge 1.
	\eq 
	Indeed, we get
 	\[
	\sum_{|\alpha|=k} \|\pa^{\alpha} f \|_{L^2}^2 
	=  \sum_{n\in \mathbb{Z}^d}\sum_{|\alpha|=k} |n^{\alpha}|^2 |\hat f(n)|^2  
	\leq d  \sum_{n\in \mathbb{Z}^d} |n|^{2k} |\hat f(n)|^2 = d\|\Lambda^k f \|_{L^2}^2
	\]
	by Parseval's theorem, and a direct calculation shows the reverse inequality
	\bq\nonumber
	\sum_{n\in \mathbb{Z}^d} |n|^{2k} |\hat f(n)|^2 \leq  
	C_{d, k} \sum_{n\in \mathbb{Z}^d} \sum_{|\alpha|=k}|n^{\alpha}|^{2} |\hat f(n)|^2 
	\eq
	from which, we derive \eqref{equiv_form}.
\end{remark}

%
%
%
%
%
%
\subsection{Main results}
Our main results are two-fold. The first result concerns the global well-posedness of the system \eqref{np_ER2}-\eqref{ini_ER2} and the large-time behavior of the constructed solutions.
\begin{theorem}\label{main_thm1}
	Let $\gamma \geq 1$, $d\ge 1$, $m \in \N$, and $\max\{d-2, 0\}\leq \alpha <d$.
	For $m>1+d/2$, suppose that the initial data $(h_0, u_0)\in H^m(\T^d) \times H^m(\T^d)$ satisfy
$\|(h_0, u_0)\|_{H^m}   \leq \varepsilon$
	for some $\varepsilon>0$ sufficiently small and $\lambda$ is small enough. Then the system \eqref{np_ER2}-\eqref{ini_ER2} admits a unique solution $(h,u) \in \mathcal{C}(\mathbb{R}_+; H^m(\T^d) \times H^m(\T^d))$.
	Moreover, there exist $C, \tilde{\mu}>0$ independent of $t$ such that 
	\bq\label{largetime}
	\|(h, u)(t)\|_{H^m} \leq Ce^{-C\tilde{\mu} t}
	\eq 
	for $t \geq 0$.
\end{theorem}

\begin{remark}\label{rmk_gc}
Since the proof for the case $\gamma = 1$ is very similar to the case $\gamma > 1$, we only present the details of the proof for the case $\gamma > 1$.
\end{remark}

\begin{remark} In the case of pressureless and repulsive Riesz interaction, the solution space $(h,u) \in H^m \times H^{m + \frac{d-\alpha}{2}}$ is considered in \cite{CJ23}. In our case, due to the presence of pressure, we cannot deal with solution space in which $h$ and $u$ have different regularities. 
\end{remark}

\begin{remark} By the Sobolev embedding, Thereom \ref{main_thm1} gives the existence and uniqueness of $\mc^1$-solutions $(h, u)$, and this together with Proposition \ref{prop_cla} below yields the existence and uniqueness of $\mc^1$-solutions $(\rho, u)$ to the system \eqref{np_ER}.
\end{remark}

Our second main result is on the a priori temporal estimates on the lowest norm of the solution $(h, u)$. In the theorem above, the large time behavior of the classical solution $(h, u)$ is obtained under the smallness assumption on initial data $(h_0, u_0)$. However, such a smallness condition is not necessary for the zeroth order norm of solutions.

\begin{theorem}\label{main_thm2}
	Let $d \ge 2$ and $(h, u)$ be a global classical solution to \eqref{np_ER2} with sufficiently small $\lambda\in (0,1)$.
	Assume that
\begin{itemize}
	\item[(i)] $\inf_{(x, t) \in \T^d \times \mathbb{R}_{+}} h(x, t) +\sigma \ge h_{min} >0$ and 
	\item[(ii)] $h, u\in L^\infty(\T^d \times \R_+)$.
\end{itemize}
	Then there exist $C, \mu>0$ independent of $t$ such that 
	\bq\label{largetime3}
	\|(h, u)(t)\|_{L^2}  \leq C e^{-C \mu t}
	\eq
	for $t \geq 0$.
\end{theorem}

\begin{remark} The lower bound assumption on $h +\sigma$ in Theorem \ref{main_thm2} corresponds to the assumption $\rho \geq \rho_{min}$ for some $\rho_{min} > 0$ in the system \eqref{np_ER}. The condition (i) in Theorem \ref{main_thm2} requires some smallness on the $L^\infty$ norm of the negative part of $h$ such that $\|h^-\|_{L^\infty} < \sigma$.
\end{remark}

\begin{remark} The constants $C$ and $\mu$ appeared in Theorem \ref{main_thm2} depend only on $\|h\|_{L^\infty}$, $\|u\|_{L^\infty}$, and $h_{min}$. Formally speaking, it does not require any differentiability of solutions with respect to the spatial variable. Thus, one may use the argument used in the proof of Theorem \ref{main_thm2} to show the large-time behavior of bounded solutions.
\end{remark}

\begin{remark}For the same reason as in Remark \ref{rmk_gc}, we only take into account the case $\gamma > 1$ for the proof of Theorem \ref{main_thm2}. When $\gamma=1$, the condition (i) in Theorem \ref{main_thm2} is not needed since $h \in L^\infty(\T^d \times \R_+)$ implies that $\rho = e^h \geq e^{-\|h\|_{L^\infty}} > 0$.
\end{remark}
 
%
%
%
%
%
\subsection{Ideas and strategies of the proof} 
Here, we give some technical ideas and strategies involved in the proofs of Theorems \ref{main_thm1} and  \ref{main_thm2}.  

\subsubsection{Global-in-time well-posedness and large-time behavior}\label{ssec_gc}

As presented, we construct our solutions in the Sobolev space $H^m$. Due to the symmetric structure of the reformulated system \eqref{np_ER2}, when we estimate $\|(h,u)\|_{H^m}$, the top order terms $(h + \sigma)(\nabla \cdot u)$ and $(h + \sigma) \nabla h$ are cancelled each other. Thus, as expected, it is important to analyze the Riesz interaction term. In the sub-Manev potential case, i.e. $\alpha \leq d-1$, roughly speaking, $\|\nabla \Lambda^{\alpha-d} \lt\{ (h+ \sigma )^{N} - \sigma^N  \rt\}\|_{\dot H^m} \ls \|h\|_{\dot H^m}$ under appropriate bound estimates, due to $m + 1 + \alpha - d \leq m$. This allows us to close the a priori estimates of solutions in $H^m$. However, in the super-Manev potential case, we get $m+1 + \alpha - d > m$, and thus the interaction term cannot be handled by a similar energy estimate as in the sub-Manev potential case. To overcome this difficulty, we develop a novel technique to reduce the order of derivatives in the Riesz interaction term with the help of pressure. To be specific, we define a function
\[
\mathcal{R}_m(k)= \lambda \sigma^{-N} \int (h+\sigma)^{k(N-2)} \Lambda^{m-kl}u\cdot \Lambda^{m-(k+2)l}\nabla \{(h+\sigma)^N -\sigma^N\} \, dx
\]
for any $0\leq k <m/l$ and $0<l<1/2$ with $l=(d-\alpha)/2$. Note that the term 
\[
\mathcal{R}_m(0) = \lambda \sigma^{-N} \int \Lambda^m u \cdot \Lambda^m \nabla \Lambda^{\alpha-d} \{ (h+\sigma)^{N} -\sigma^{N}\} \, dx
\]
appears in the estimate of $\|(h,u)\|_{\dot H^m}$. By an iteration argument, for $1\leq k_0 <m/l$, we can derive 
\[\begin{aligned}
\mathcal{R}_m(0) &\leq C \|(h,u)\|_{\dot H^m}^2   - \sum_{k=1}^{k_0}\frac{\tilde{\lambda}^k}{2} \frac{d}{dt} \int (h+\sigma)^{k(N-2)} |\Lambda^{m-kl} u|^2 \, dx,
\end{aligned}\]
where $\tilde{\lambda}=\lambda \sigma^{-N}N^2$ and $C = C(\lambda, \|(h, u)\|_{H^m}) > 0$ independent of $t$. In particular, here $C \to 0$ as $\lambda \to 0$. Then we can close the $H^m$ a priori estimate for the solution. On the other hand, for the global-in-time regularity of solutions, it is required to have appropriate dissipation rates for $h$ and $u$. The linear velocity damping directly gives the dissipation term $\|u\|_{H^m}$ in the estimate of $\|(h, u)\|_{H^m}$, however, the dissipation rate for $h$ is not simply obtained from that estimate. For this, we employ a hypocoercivity-type estimate. We take into account the following integral: 
\[
\int \nabla (\Lambda^s h) \cdot \Lambda^s u\, dx
\]
with $s \leq m-1$ which gives the dissipation rate for $\nabla h$ in $H^{m-1}$. Note that for $\mu > 0$ small enough 
\[
\|(h,u)\|_{H^m}^2  + \mu\sum_{s=0}^{m-1}\int \nabla (\Lambda^s h) \cdot \Lambda^s u\, dx \sim \|(h,u)\|_{H^m}^2. 
\]
Then, together with adding the zeroth-order estimate of solutions and taking $\lambda > 0$ small enough, we derive 
\[
\frac{d}{dt}{\mathcal X}_m + \xi {\mathcal X}_m \leq 0
\]
for some $\xi > 0$, where 
\[
{\mathcal X}_m := \sum_{k=0}^{k_0} \tilde{\lambda}^k \int  (h+\sigma)^{k(N-2)} |\Lambda^{m-kl} u|^2\, dx
+\|u\|_{L^2}^2
+\|\Lambda^m h\|_{L^2}^2
+\|h\|_{L^2}^2
+\mu \int \nabla (\Lambda^{m-1} h) \cdot \Lambda^{m-1} u\, dx.
\]
Since ${\mathcal X}_m \sim  \|(h,u)\|_{H^m}^2$, we can have the uniform-in-time bound on $\|(h,u)(t)\|_{H^m}^2$, and moreover, it decays to zero exponentially fast as time goes to infinity.

\subsubsection{A priori large-time behavior} The proof of the large-time behavior estimate relies on the Lyapunov function method. To this end, we first consider the following modulated energy function $\mathcal{E} = \mathcal{E}(t)$:
\bq\nonumber
\begin{aligned}
	\mathcal{E} := & \int (h+\sigma)^N |u-m_c|^2\, dx 
	+\frac{2}{N(N+2)} \int f\lt(\, \frac{N+2}{N}, (h+\sigma)^N; \sigma^N \, \rt) dx \\
	& -\frac{\lambda }{\sigma^{N}} \int | \Lambda^{\frac{\alpha-d}{2}} \{ (h+\sigma)^N -\sigma^N\}|^2\, dx+ \frac{1}{2}|m_c|^2,
\end{aligned}
\eq  
where $f(\gamma, r; r_0)$ and $m_c = m_c(t)$ are given by
\[
f(\gamma, r; r_0) := r \int_{r_0}^r \frac{s^{\gamma}-r_0^{\gamma}}{s^2}\, ds
\]
and
\bq\nonumber
m_c :=  \sigma^{-N} \int (h+\sigma)^N u\, dx,
\eq 
respectively, for $r_0 > 0$, $\gamma \geq 1$, and $r\in [0, \tilde{r}]$ with $\tilde r > 0$. Then, we can show that $\mathcal{E}$ satisfies
	\bq\nonumber
	\frac{d}{dt} \mathcal{E} +\mathcal{D} =0,
	\eq
	where the dissipation rate $\mathcal{D} = \mathcal{D}(t)$ is given by
	\bq \nonumber
	\mathcal{D}:=2\nu \int (h+\sigma)^N |u-m_c|^2\, dx + \nu |m_c|^2.
	\eq
On the other hand, it follows from the symmetry of the operator $\Lambda^{\alpha-d}$ that 
\[
m_c '(t) =- \nu m_c(t), 
\]
and thus
\[
m_c(t) \leq m_c(0)e^{-\nu t}.
\]
Moreover, under the assumptions of Theorem \ref{main_thm2}, we can obtain
\[
\mathcal{E}(t)\sim \|(h,u)(t)\|_{L^2}^2 + |m_c(t)|^2 \quad \mbox{and} \quad \mathcal{D}(t) \sim \|u(t)\|_{L^2}^2 + |m_c(t)|^2.
\]
Here we would like to remark that there is no dissipation rate for $h$ in $\mathcal{D}$. Thus, to get the exponential decay of convergence of $\|(h,u)(t)\|_{L^2}^2$, we need to get some dissipation term $\|h\|_{L^2}^2$ in $\mathcal{D}$. Precisely, it suffices to have 
\[
\mathcal{D}(t) \gs \|(h,u)(t)\|_{L^2}^2 + |m_c(t)|^2.
\]
For this, motivated from \cite{CJ21}, we take into account a perturbed energy function $\mathcal{E}^{\mu} = \mathcal{E}^{\mu}(t)$: 
\bq\nonumber
\mathcal{E}^{\mu} := \mathcal{E} + \mu \int  (u-m_c) \cdot \nabla W*[(h+\sigma)^N -\sigma^N]\,dx
\eq 
for sufficiently small $\mu>0$, where the potential $W$ is an even function explicitly written as
\bq\label{eqn_W}
W(x) = \left\{\begin{array}{lcl}\displaystyle  -c_0 \log |x| + G_0(x) & \mbox{ when } & d=2, \\
c_1|x|^{2-d} + G_1(x) & \mbox{ when } & d\ge 3.\end{array}\right.
\eq
Here $c_0>0$ and $c_1>0$ are normalization constants and $G_0$ and $G_1$ are smooth functions over $\T^2$ and $\T^d$ ($d\ge 3$), respectively. 
We notice that
\[
\mathcal{E}(t) \sim \mathcal{E}^{\mu}(t)
\]
for $\mu > 0$ small enough, and it is clear to get 
\bq\label{e_pertu}
\frac{d}{dt} \mathcal{E}^{\mu}(t) +\mathcal{D}^\mu(t) =0,
\eq
where 
\[
\mathcal{D}^{\mu}(t):= \mathcal{D}(t) - \mu \frac{d}{dt}\int  (u-m_c) \cdot \nabla W*[(h+\sigma)^N -\sigma^N]\,dx.
\]
We then observe that the perturbed dissipation rate $\mathcal{D}^\mu$ produces the required dissipation rate for $h$, in fact, we can show 
\[
\mathcal{D}^\mu(t) \gs \|(h,u)(t)\|_{L^2}^2 + |m_c(t)|^2,
\]
from which we have the result of Theorem \ref{main_thm2}.

\begin{remark} For any $g \in L^2(\T^d)$ with $\int  g \,dx =0$, $U := W* g \in H^1(\T^d)$ is the unique function that satisfies the following condition:
\[
\int  U\,dx = 0 \quad \mbox{and} \quad \int  \nabla U \cdot \nabla \psi\,dx = \int  g\,\psi\,dx \quad \forall\, \psi \in H^1(\T^d),
\]
i.e. $U$ is the unique weak solution to $-\Delta U = g$.
\end{remark}

%
%
%
%
%
\subsection{Organization of paper} The rest of this paper is organized as follows. In Section \ref{sec:2}, we present preliminary materials used throughout the paper. In Section \ref{sec:3}, we provide a priori estimates for system \eqref{np_ER2} showing the uniform-in-time bound estimates of our solutions $(h,u)$ in $H^m \times H^m$ space under the smallness assumptions on the solutions. Section \ref{sec:4} is devoted to proving Theorem \ref{main_thm1} based on the estimates from the previous section. Finally, in Section \ref{sec:5}, we present details of the proof for Theorem \ref{main_thm2}.

%
%
%
%
%

\section{Preliminaries}\label{sec:2}
\setcounter{equation}{0}
In this section, we present local-in-time well-posedness theory for the system \eqref{np_ER2} and list several useful inequalities which will be used for obtaining a priori estimates in the following sections.

\subsection{Local well-posedness}
We first provide the local-in-time existence and uniqueness of regular solutions to the system \eqref{np_ER2}-\eqref{ini_ER2}. 
  
\begin{theorem}\label{local}
	Let $d\ge1$ and $\max\{d-2, 0 \}\leq\alpha<d$.
	For any $m>1+d/2$, suppose that $(h_0, u_0)\in H^{m}(\T^d)\times H^{m}(\T^d)$.
	Then for any positive constants $\varepsilon_1<M_0$, there exists a positive constant $T_0>0$ depending only on $\varepsilon_1, M_0$ such that if $\|(h_0, u_0)\|_{H^m}
	<\varepsilon_1$, then the system \eqref{np_ER2}-\eqref{ini_ER2} admits a unique solution $(h, u)\in \mathcal{C}([0, T_0); H^m(\T^d)\times H^m(\T^d))$ satisfying 
	$$
	\sup_{0\leq t\leq T_0} \|(h, u)(t)\|_{H^m} \leq M_0.
	$$
\end{theorem}

The above local well-posedness can be established by using a similar argument as in \cite{CJe22}, where the Euler-Riesz system is studied when $\rho$ is integrable. 

In the proposition below, we show a relation between two solutions $(\rho, u)$ and $(h,u)$ to the systems \eqref{np_ER} and \eqref{np_ER2}, respectively.


\begin{proposition}\label{prop_cla}
	Let $d\ge1$ and $\gamma>1$. 
	For any fixed $T>0$ and given $\sigma=\gamma^{1/2}N$, if $(\rho, u)\in \mathcal{C}^1(\T^d \times [0, T))$ solves the system \eqref{np_ER} with $\rho>0$, 
	then $(h, u)\in \mathcal{C}^1(\T^d \times [0, T))$ solves the system \eqref{np_ER2} with $h+\sigma>0$.
	Conversely, if $(h, u)\in \mathcal{C}^1(\T^d \times [0, T))$ solves the system \eqref{np_ER2} with $h+\sigma>0$, then $(\rho, u)\in \mathcal{C}^1(\T^d \times [0, T))$ solves the system \eqref{np_ER} with $\rho>0$.
\end{proposition}

In a straightforward way, we obtain the above proposition, where the positivity of the density $\rho$ is obtained from the corresponding positivity of the density  $(h + \sigma)$ by using the method of characteristics. We refer to \cite{STW} for more details. 

\subsection{Useful inequalities}
In this part, we present the Moser-type and Sobolev-type inequalities in the following lemmas. 
\begin{lemma}\label{lem_moser}
The following relations hold. 
\begin{enumerate}
\item[(i)]
	Let $k\in \mathbb{N}$. If $f, g \in H^k\cap L^{\infty}(\T^d)$ then
	\bq\label{moser}
	\|\Lambda^k (fg) \|_{L^2} \leq C_{d, k} ( \|f\|_{L^{\infty}}\|\Lambda^k g\|_{L^2} +\| g\|_{L^{\infty}}\|\Lambda^k f\|_{L^2}).
	\eq
Moreover, if $f\in \mathcal{C}^k(\T^d)$, then there exists a positive constant $C=C(k, f)$ such that
	\bq\label{chain rule}
	\|\Lambda^k (f \circ g) \|_{L^{2}} \leq C  \|g\|_{L^{\infty}}^{k-1} \|\Lambda^k g\|_{L^{2}}.
	\eq
\item[(ii)]
Let $s \in \mathbb{R}$. For any $\varepsilon>0$, there exist $C_{s, \varepsilon}>0$ depending only on $s, \varepsilon$ such that 
	\begin{equation}\label{comm}
		\|[\Lambda^s \partial_{x_i}, g] f\|_{L^2} 
		\leq C_{s, \varepsilon} (\|g\|_{H^{\frac{d}{2}+1+\varepsilon}} \|\Lambda^s f\|_{L^2} +\|g\|_{H^{\frac{d}{2}+1+s+\varepsilon}} \|f\|_{L^2})
	\end{equation}
and 
\begin{equation}\label{comm0}
	\|[\Lambda^s, g] f\|_{L^2} 
	\leq C_{s, \varepsilon} (  \| g \|_{H^{\frac{d}{2}+1+\varepsilon}} \| \Lambda^{s-1} f \|_{L^2}
	+  \|f\|_{H^{\frac{d}{2}+\varepsilon}} \| \Lambda^s g\|_{L^2}).
\end{equation}
Here, $[\cdot, \, \cdot]$  denotes the commutator operator, i.e. $[A, B]:=AB-BA$.
\end{enumerate}
\end{lemma}
\begin{proof} Since the inequality \eqref{moser} is classical, we only deal with the commutator estimates appeared in (ii). The inequality \eqref{comm} is obtained in \cite{CCCGW} for the case $d=2$, and it can be readily checked that
it is also available for any $d\ge1$. On the other hand, the proof of \eqref{comm0} follows a similar way to that in \cite{CCCGW}, however, for completeness, we give details of the proof of \eqref{comm0}. 

By definition, we get
	\bq\label{fourier}
	\widehat{[\Lambda^s, g]f}(n) = \sum_{m\in \mathbb{Z}^d}  (|n|^s-|n-m|^s ) \hat{f}(n-m) \hat{g}(m), \quad n \in \Z^d.
	\eq 
We also observe  
	\bq\label{diff}
	\begin{aligned}
	 ||n|^s-|n-m|^s|  
	&= \lt| \int_0^1 \frac{d}{dr}|n r + (n-m) (1-r)|^s  \rt|  \leq |s|  \int_{0}^1 |n r +(n -m)(1-r)|^{s-1} |m|\,  dr.  \\
	\end{aligned}
	 \eq 
If $s \ge 1$, we get
\bq\nonumber
\begin{aligned}
	\lt| |n|^s-|n-m|^s \rt|  
	&\leq s \lt(|n|^{s-1}+|n-m|^{s-1} \rt)|m| \leq \max\{s, 2^{s-2}s \} (|n-m|^{s-1}|m|+|m|^{s})
\end{aligned}
\eq 
due to
$
	|n|^{s-1} \leq \max\{1, 2^{s-2}\} (|n-m|^{s-1}+|m|^{s-1}).
$ 
By combining this with \eqref{fourier}, we get
	\bq\label{fo}
	\begin{aligned}
		\lt| \widehat{[\Lambda^s, g]f}{(n)}\rt|
		&\leq 
		  \sum_{m\in \mathbb{Z}^d}  \big| |n|^s-|n-m|^s \big|\, \big|\hat{f}(n-m)\big|\, \big| \hat{g}(m)\big|  \\
		&\leq 
		C \sum_{m\in \mathbb{Z}^d} \big||n-m|^{s-1}\hat{f}(n-m)\big|\,\big| m \hat{g}(m)\big|  
		+ C  \sum_{m\in \mathbb{Z}^d}  \big|\hat{f}(n-m)\big|\, \big| |m|^s \hat{g}(m)\big| \\
		&\leq 
		C  \sum_{m\in \mathbb{Z}^d} \big|\widehat{ \Lambda^{s-1} f}(n-m)\big| \,\big|  \widehat{\nabla  g}(m)\big| 
		+ C \sum_{m\in \mathbb{Z}^d} \big|\hat{f}(n-m)\big|\, \big| \widehat{ \Lambda^s g}(m)\big|\\
		&=
		C  ( |\widehat{ \Lambda^{s-1} f}| \ast |  \widehat{\nabla  g}| )(n)
		+ C (|\hat{f}| \ast | \widehat{ \Lambda^s g}|)(n).
	\end{aligned}
	\eq
By applying the Parseval theorem, \eqref{fo} and then Young's inequality in turn, 
we deduce that 
\bq\nonumber
\begin{aligned}
\| [\Lambda^s, g]f \|_{L^2} =\| \widehat{[\Lambda^s, g]f} \|_{l^2}
&\leq 	C  \|\widehat{ \Lambda^{s-1} f} \|_{l^2} \| \widehat{\nabla  g} \|_{l^1} 
	+ C \|\hat{f}\|_{l^1} \| \widehat{ \Lambda^s g}\|_{l^2}.
\end{aligned}
\eq
It is easy to check that 
	\bq\nonumber
	\begin{aligned}
	 \| \widehat{ \Lambda^{k} f}(n) \|_{l^1} 
	=\sum_{n\in\mathbb{Z}^d} |n|^k \big|\widehat{f}(n) \big| 
	\leq \Biggl(\sum_{n\in\mathbb{Z}^d} \frac{1}{(1+|n|^2)^{\frac{d}{2}+\varepsilon}} \Biggl)^{\frac{1}{2}} \Biggl( \sum_{n\in\mathbb{Z}^d} |(1+|n|^2)|^{\frac{d}{2}+k+\varepsilon} \big| \widehat{f}(n) \big|^2 \Biggl)^{\frac{1}{2}}
	\leq C_{\varepsilon} \|f\|_{H^{\frac{d}{2}+k+\varepsilon}}
	\end{aligned}
	\eq
	for any $k>0$. Thus, we use the Parseval theorem to obtain 
	  \bq\nonumber
	  \begin{aligned}
	   \| [\Lambda^s, g]f\|_{L^2} 
	  	\leq 
	  	C_{s, \varepsilon}  \| \Lambda^{s-1} f \|_{L^2} \| g \|_{H^{\frac{d}{2}+1+\varepsilon}} 
	  	+ C_{s, \varepsilon} \|f\|_{H^{\frac{d}{2}+\varepsilon}} \| \Lambda^s g\|_{L^2}.
	  \end{aligned}
	  \eq
On the other hand, for $s<1$, from \eqref{diff} we see that
	\bq\nonumber
\begin{aligned}
	\lt| \widehat{[\Lambda^s, g]f}{(n)}\rt|
	&\leq 
	\sum_{m\in \mathbb{Z}^d}  \big| |n|^s-|n-m|^s \big|\, \big|\hat{f}(n-m)\big|\, \big| \hat{g}(m)\big|\\
	& \leq  \sum_{m\in \mathbb{Z}^d}  
	|s|  \int_{0}^1 |n r +(n -m)(1-r)|^{s-1} \, \big|\hat{f}(n-m)\big|\, \big| m \hat{g}(m)\big| \,dr.
\end{aligned}
\eq
Here we note for $r\in [0, 1]$ that 
\bq\nonumber
\begin{aligned}
|nr +(n-m)(1-r)|^{s-1} 
&= |n-m + rm |^{s-1} \\
&=(1+r)^{s-1} \Big| \frac{n-m}{1+r} + \frac{mr }{1+r} \, \Big|^{s-1}\\
&\leq (1+r)^{s-1}  \lt( \frac{|n-m|^{s-1}}{1+r} + \frac{|m|^{s-1}r}{1+r} \, \rt)\\
&= (1+r)^{s-2}  \lt( |n-m|^{s-1} + |m|^{s-1}r \, \rt)
\end{aligned}
\eq 
from the fact that $|\cdot|^{s-1}$ is convex for $s<1$.
Then it follows that 
\bq\nonumber
\begin{aligned}
	\lt|\widehat{[\Lambda^s, g]f}{(n)}\rt|
	& \leq  \sum_{m\in \mathbb{Z}^d}  
	|s|  \, \big| |n-m|^{s-1} \hat{f}(n-m)\big|\, \big| m \hat{g}(m)\big|  \int_{0}^1 (1+r)^{s-2}\,dr \\
	&\quad + \sum_{m\in \mathbb{Z}^d}  
	|s|  \, \big|\hat{f}(n-m)\big|\, \big| |m|^s \hat{g}(m)\big|  \int_{0}^1 r(1+r)^{s-2}\,dr, 
\end{aligned}
\eq
which implies the second row in \eqref{fo}. This completes the proof.
\end{proof}

We next present some useful Sobolev type inequality on $\T^d$.
\begin{lemma}\label{Sobolev_T}
Let $s_1, s_2 \in \mathbb{R}$.
If $0< s_1\leq s_2$, then we have
	\bq\label{sobolev_t}
	\|\Lambda^{s_1} f\|_{L^2} \leq \|\Lambda^{s_2}f \|_{L^2}.
	\eq 	
Moreover, if  
$\int  f\, dx =0$, then \eqref{sobolev_t} holds for any $s_1\leq s_2$.
\end{lemma}
\begin{proof}
The inequality \eqref{sobolev_t} can be easily obtained as follows:
\bq\label{so_t}
\|\Lambda^{s_1} f\|_{L^2}^2
= \sum_{n\in \mathbb{Z}^d} |n|^{2(s_1-s_2)} |n|^{2s_2} |\hat f(n)|^2 
\leq \sum_{n\in \mathbb{Z}^d} |n|^{2s_2}  |\hat f(n)|^2 
= \| \Lambda^{s_2}f \|_{L^2}^2.
\eq
If $\int  f\, dx =0$, then
the case $s_1=0$ immediately follows from the Poincar\'e inequality.
Moreover, $\int  f\, dx =0$ implies $\hat{f}(n)=0$ when $n=0$ and then 
$\sum_{n \in \mathbb{Z}^d}$ in \eqref{so_t} can be replaced by $\sum_{n \in \mathbb{Z}^d/\{0\}}$
so that \eqref{so_t} is valid for any $s_1\leq s_2$.  
\end{proof}

%
%
%
%
%

\section{A priori estimates}\label{sec:3}
\setcounter{equation}{0}

In this section, we provide a priori estimates of solutions for the system \eqref{np_ER2}-\eqref{ini_ER2} in $\mathcal{C}([0, T); H^m(\T^d) \times H^m(\T^d))$. 

We start with $L^2$-estimates for the pair $(h, u)$.
\begin{lemma}\label{zero_u}
	Let $d\ge1$, $\gamma > 1$, and $m>1+d/2$ with $m \in \N$. For $T>0$, let $(h, u)\in \mathcal{C}([0, T); H^m(\T^d) \times H^m(\T^d))$ be a classical solution to system \eqref{np_ER2} satisfying
	\bq\label{as}
	\sup_{0\leq t\leq T}\|h(t)\|_{L^\infty} \leq \frac{\sigma}{2}.
	\eq 
	Then we have
	\bq\nonumber
	\begin{aligned}
&\frac{1}{2}\frac{d}{dt} \left( \int |h|^2+|u|^2\, dx\right)
+\nu \int |u|^2 \, dx\cr
&\quad \leq \lt(C\|h\|_{H^m} + C\|u\|_{H^m}+ C_0 \lambda \rt)\|u\|_{L^2}^2 +\lt(C\|h\|_{H^m}+ C_0 \lambda  \rt)\|\nabla h\|_{L^2}^2,
	\end{aligned}
	\eq 
where 
\bq\label{c_0}
C_0= C_d (2\sigma)^{-1} N \max \{2^{-(N-1)}, 2^{N-1} \}
\eq
with $C_d > 1$.
Here, $C$ is a positive constant depending on $ m$ and $\gamma$.
\end{lemma}
\begin{proof}
We first see that 
\bq\nonumber
\begin{aligned}
&\frac{1}{2}\frac{d}{dt} \left( \int |h|^2 +|u|^2 \, dx\right)  +\nu \int |u|^2\, dx\cr
&\quad = -\int h u\cdot \nabla h\, dx -\int |u|^2 \nabla \cdot u\, dx -\frac{1}{N} \int h(h+\sigma) (\nabla \cdot u)\, dx -\frac{1}{N}\int u \cdot (h+\sigma) \nabla h\, dx \\
&\qquad +\lambda \sigma^{-N} \int u \cdot \nabla \Lambda^{\alpha-d}  \lt\{ (h+ \sigma )^{N} - \sigma^N  \rt\}  dx\\
&\quad =:\sum_{i=1}^4 \sfI_i+\mathcal{R}_0,
\end{aligned}
\eq
where we estimate
\bq\nonumber
\sfI_1 \leq \|h\|_{L^{\infty}} \|u\|_{L^2}\|\nabla h\|_{L^2} \quad \text{and} \quad \sfI_2  \leq \|\nabla \cdot u\|_{L^{\infty}} \|u\|_{L^2}^2.
\eq 
By using the integration by parts, we also get
\bq\nonumber
\begin{aligned}
\sfI_3 + \sfI_4 
=\frac{1}{N}\int h \nabla h  \cdot u\, dx \leq \frac{1}{N} \|h\|_{L^{\infty}} \| \nabla h\|_{L^2} \|u\|_{L^2}. 
\end{aligned}
\eq 
To control the term $\mathcal{R}_0$, we notice that  \eqref{as} implies 
\[
\begin{aligned}
	\frac{\sigma}{2} \leq	h(t, x)+ \sigma \leq \frac{3\sigma}{2} \quad \mbox{for all } (x, t) \in \T^d \times [0,T].
\end{aligned}
\] 
This, together with Lemma \ref{Sobolev_T} for $\alpha-d\leq 0$, gives
\begin{align*}
 \mathcal{R}_0 
&\leq \lambda \sigma^{-N}\|u \|_{L^2} \| \Lambda^{\alpha-d} \nabla   \{ (h+ \sigma )^{N} - \sigma^{N} \} \|_{L^2}\\
&\leq \lambda \sigma^{-N}\|u \|_{L^2} \| \nabla  \{ (h+ \sigma )^{N} - \sigma^N \} \|_{L^2}\\
&=  \lambda \sigma^{-N} N\|u \|_{L^2} \|(h+\sigma)^{N-1} \nabla h \|_{L^2}\\
&\leq  \sigma^{-1} N \max \{2^{-(N-1)}, 2^{N-1} \} \lambda  \|u \|_{L^2} \| \nabla h \|_{L^2}\\
&\leq  2 C_0 \lambda  \|u \|_{L^2} \| \nabla h \|_{L^2}
\end{align*}
for $C_0>0$ given as in \eqref{c_0}
Then, we use Young's inequality to deduce
\bq\nonumber
\begin{aligned}
&\frac{1}{2}\frac{d}{dt} \left( \int |h|^2 +|u|^2 \, dx\right)
+\nu \int |u|^2 \, dx\\
& \quad \leq 
C\|h\|_{L^{\infty}}\|\nabla h\|_{L^2}\|u\|_{L^2}
+ \|\nabla \cdot u\|_{L^{\infty}}\|u\|_{L^2}^2   +  2C_0 \lambda \|u\|_{L^2} \|\nabla h\|_{L^2}\\
& \quad \leq 
\lt(C\|h\|_{L^{\infty}} + \|\nabla \cdot u\|_{L^{\infty}}+ C_0 \lambda \rt)\|u\|_{L^2}^2
+\lt(C\|h\|_{L^{\infty}}+ C_0 \lambda \rt)\|\nabla h\|_{L^2}^2, 	
\end{aligned}
\eq 
where $C$ is a positive constant depending only on $\gamma$.
Finally, we use
\bq\nonumber
\|\nabla \cdot u\|_{L^{\infty}} 
\leq C_m\|u\|_{H^{m}},\quad
\|h\|_{L^{\infty}} \leq C_m\|h\|_{H^{m-1}} \leq C_m\|h\|_{H^{m}}
\eq
to obtain the desired estimates.
\end{proof}

Before proceeding further, we present a technical lemma and a remark on the derivative estimates of $h$.
\begin{lemma}\label{tool}
	Let $d\ge1$ and $m>d/2+1$. 
	For given $T>0$ and $\sigma>0$, let $h\in \mathcal{C}([0, T); H^m(\T^d) )$ with
	\[
	\sup_{0\leq t\leq T}\|h(t)\|_{L^\infty} \leq \frac{\sigma}{2}.
	\]
	Then there exists a positive constant $C_{d,k}$ such that
	\bq\label{chain2}
	\|\Lambda^k (h+\sigma)^\beta\|_{L^2} \leq  
	\max\lt\{ (\sigma/2)^{\beta-1}, (2\sigma)^{\beta-1}\rt\} |\beta| \|\Lambda^k h\|_{L^2} + C_{d,k}(1-\delta_{1,k}) \|\Lambda^{k} h\|_{L^2}\|h\|_{H^m}(1+\|h\|_{H^m})^{k-2}
	\eq 
	for any integer $k\leq m$ and $\beta \in \mathbb{R}$, where $\delta_{ij}$ denotes the Kronecker's delta.
\end{lemma}
\begin{proof}
We proceed by induction. When $k=1$, we directly have
\[
\|\nabla (h+\sigma)^\beta\|_{L^2} \le \max\lt\{ (\sigma/2)^{\beta-1}, (2\sigma)^{\beta-1}\rt\}|\beta| \|\nabla h\|_{L^2}.
\]
Now, suppose that \eqref{chain2} holds for $ k \ge 1$. Then
\[
\nabla^{k+1} (h+\sigma)^\beta = \beta (h+\sigma)^{\beta-1} \nabla^{k+1} h + \beta[\nabla^k, (h+\sigma)^{\beta-1}] \nabla h.
\]
Thus, we use Lemma \ref{lem_moser} and induction hypothesis to deduce
\bq\nonumber
\begin{aligned}
	\|\nabla^{k+1} (h+\sigma)^\beta \|_{L^2} 
	&\leq  \max\lt\{ (\sigma/2)^{\beta-1}, (2\sigma)^{\beta-1}\rt\} |\beta| \|\nabla^{k+1} h\|_{L^2} + C\|[\nabla^k, (h+\sigma)^{\beta-1}] \nabla h\|_{L^2} \\
	&\leq \max\lt\{ (\sigma/2)^{\beta-1}, (2\sigma)^{\beta-1}\rt\} |\beta| \|\nabla^{k+1} h\|_{L^2}   \\
	&\quad +C\lt(\|\nabla h\|_{L^\infty}\|\nabla^k(h+\sigma)^{\beta-1}\|_{L^2} + \|\nabla (h+\sigma)^{\beta-1}\|_{L^\infty}\|\nabla^{k}h\|_{L^2} \rt)\\
	&\leq \max\lt\{ (\sigma/2)^{\beta-1}, (2\sigma)^{\beta-1}\rt\} |\beta| \|\nabla^{k+1} h\|_{L^2} + C \|\Lambda^{k+1} h\|_{L^2}\|h\|_{H^m}(1+\|h\|_{H^m})^{k-1},
	\end{aligned}
\eq 
and this yields the desired estimate.
\end{proof}

\begin{remark}\label{rem2}
Suppose that 
	\[
	\sup_{0\leq t\leq T}\|h(t)\|_{L^\infty} \leq \frac{\sigma}{2}
	\]
	holds  for some $T>0$. Then we find
	\[
	\|(h+\sigma)^n - \sigma^n\|_{L^2} \sim \|h\|_{L^2}
	\]
	for any $n > 0$. 
	Since Lemma \ref{Sobolev_T} implies
	\[
	\|h\|_{H^m} \sim \|h\|_{L^2} + \|\Lambda^m h\|_{L^2},
	\] 
	with the help of Poincar\'e inequality, we obtain
	\[\begin{aligned}
	\|h\|_{H^m} &\sim \|(h+\sigma)^N - \sigma^N\|_{L^2} + \|\Lambda^m h\|_{L^2} \\
	&\lesssim \|\nabla (h+\sigma)^N\|_{L^2}+ \|\Lambda^m h\|_{L^2}\\
	&\lesssim \|\nabla h\|_{L^2}+ \|\Lambda^m h\|_{L^2} \\
	&\lesssim \|\Lambda^m h\|_{L^2} 
	\end{aligned}\]
	for $t \in [0, T]$. This shows
	\[
	\|h\|_{H^m }\sim \|\Lambda^m h\|_{L^2}.
	\]
\end{remark}

Now, we provide the estimates of higher-order derivatives of solutions.  Thanks to Remark \ref{rem2}, it suffices to deal with the highest order.
\begin{proposition}\label{higher_u}
Let the assumptions of Lemma \ref{zero_u} be satisfied. Then, depending on $\alpha$, we have the following estimates:
\begin{itemize}
	\item[(i)] (Sub-Manev potential case) If $-2\leq \alpha-d\leq -1$, then we have
\bq\label{h1}
\begin{aligned}
	&\frac{1}{2} \frac{d}{dt}\left(
	\int |\Lambda^{m} h|^2 \, dx+ \int  |\Lambda^{m} u|^2\, dx\right) +\nu \int  |\Lambda^{m} u|^2 \, dx \\
	&\quad \leq ( C\|h\|_{H^m}+ C\|u\|_{H^m} + C_0\lambda +C \lambda \|h\|_{H^m}(1+\|h\|_{H^m})^{m-2}) \|\Lambda^{m} u\|_{L^2}^2\\ 
	&\qquad +\lt(C \|h\|_{H^m}+ C\|u\|_{H^m}+ C_0\lambda +C \lambda \|h\|_{H^m}(1+\|h\|_{H^m})^{m-2}\rt)\|\Lambda^{m} h\|_{L^2}^2,
\end{aligned}
\eq 
where 
\[
C_0 := C_d (2\sigma)^{-1} N \max \{2^{-(N-1)}, 2^{N-1} \}  \quad \mbox{with} \quad C_d>1.
\]
\item[(ii)] (Super-Manev potential case) If $-1<\alpha-d <0$, then we have
\bq\label{h2}
\begin{aligned}
	&\frac{1}{2} \frac{d}{dt}\left(
	\int |\Lambda^{m} h|^2 \, dx+ \sum_{k=0}^{k_0}\tilde{\lambda}^k \int  (h+\sigma)^{k(N-2)} |\Lambda^{m-kl} u|^2\, dx\right) +\nu \int  |\Lambda^{m} u|^2 \, dx \\
	&  \leq  C\lt((1+\hat\lambda)(\|u\|_{H^m}+\|h\|_{H^m})+\hat\lambda \|u\|_{H^m}\|h\|_{H^m}+\lambda^{k_0+1}\|h \|_{H^m}(1+\|h\|_{H^m})^{m-2} \rt)   \| \Lambda^m u\|_{L^2}^2 \\
	&\quad 
	 +	C \lt( \|u\|_{H^m}+   \|h\|_{H^m} +  \hat\lambda \|h\|_{H^m} +\lambda^{k_0+1}(1+\|h\|_{H^m})^{m-2}  \|h\|_{H^m}) \rt)  \|\Lambda^m  h\|_{L^2}^2 \\
	 &\quad + C_{1}\lambda^{k_0 + 1}  \| \Lambda^m u\|_{L^2}^2 +C_{1}\lambda^{k_0 + 1}\|\Lambda^m  h\|_{L^2}^2 
\end{aligned}
\eq 
where 
\[
\begin{aligned}
C_{1} := C_d  \sigma^{-(2k_0 +1)} &N^{2k_0+1}  \max\{ 2^{-k_0 (N-2)-N}, 2^{-(k_0-1) (N-2)},  2^{(k_0+1) (N-2)}\},\\
& \tilde{\lambda} := \lambda \sigma^{-N} N^2,\quad \hat\lambda  := \sum_{k=1}^{k_0} \lambda^k,
\end{aligned}
\] 
and $k_0\in \mathbb{N}$ is a minimum value of $k$ such that 
\[
\max\lt\{ 1,\,  \frac2{d-\alpha} -2 \rt\} \leq k < \frac{2m}{d-\alpha}.
\]
\end{itemize}
Here $C=C(\gamma, m, \alpha, d)$ is a positive constant.
\end{proposition}
 
For the proof of Proposition \ref{higher_u}, a crucial part is to control the following Riesz interaction term:
\[
 \lambda \sigma^{-N} \int \Lambda^m u \cdot \Lambda^m \nabla \Lambda^{\alpha-d} \{ (h+\sigma)^{N} -\sigma^{N}\} \, dx.
\]
As mentioned in Section \ref{ssec_gc}, when $-2<\alpha-d\leq -1$ (sub-Manev potential case), the total derivative order of the operator $\Lambda^{m+\alpha-d} \nabla$ is roughly less than or equal to $m$, and hence,  using Lemma \ref{Sobolev_T} and \ref{tool} gives the desired result.
However,  it does not work directly for the case $-1<\alpha-d <0$ (super-Manev potential case).
To handle this issue, we shall use an iteration scheme to reduce the order of the derivative in the Riesz interaction. Let us recall 
\bq\label{R_k}
\mathcal{R}_m(k)= \lambda \sigma^{-N} \int (h+\sigma)^{k(N-2)} \Lambda^{m-kl}u\cdot \Lambda^{m-(k+2)l}\nabla \{(h+\sigma)^N -\sigma^N\} \, dx
\eq
and consider a function
\[
\mathcal{P}_{m}(k):=\frac{1}{N}\int (h+\sigma)^{k(N-2)}\Lambda^{m-kl} u \cdot \Lambda^{m-kl} \{ (h+\sigma) \nabla h\}\, dx
\]
for any $0\leq k <m/l$ and $0<l<1/2$, where $l=(d-\alpha)/2$.

In the following lemma, we show a relation between  $\mathcal{P}_m(k)$ and $\mathcal{R}_{m}(k)$ for any $0\leq k <m/l$ whose proof is postponed to Appendix \ref{app.A} for the smoothness of reading.
\begin{lemma} \label{higher_it}
Let the assumptions of Lemma \ref{zero_u} be satisfied and $0<l<1/2$. Then, for any integer $k \in [1, \, m/l)$, we have
	\bq\label{estimate_Pk}
	\begin{aligned}
		\mathcal{P}_{m}(k)   
		& \leq \mathcal{R}_m (k) +
		C (1+\|h\|_{H^m}) \|u\|_{H^m}\|\Lambda^{m} u \|_{L^2}^2 \nonumber
		-\frac{1}{2} \frac{d}{dt}\int (h+\sigma)^{k(N-2)} |\Lambda^{m-kl} u|^2 \, dx, 
	\end{aligned}
	\eq 
	where $C=C(\gamma, m, k, l)$ is a positive constant.
\end{lemma}

Together with the lemma above, the following is the core of the proof for Proposition \ref{higher_u}. By using them together, we can manipulate the term $\mathcal{R}_m(k)$ into the term $\mathcal{R}_m(k+1)$ with reduced order of the derivative on the Riesz interaction term.

\begin{lemma}\label{Gap}
Let the assumptions of Lemma \ref{zero_u} be satisfied and $0<l<1/2$. Then, for any integer $k \in [0,\, m/l-1)$, we have
\bq\label{gap}
\begin{aligned}
	\lt|  \mathcal{R}_{m}(k) - \tilde{\lambda} \mathcal{P}_{m}(k+1) \rt|
	\leq C \lambda \|h\|_{H^m}\| \Lambda^{m} u\|_{L^2}^2+ C \lambda \|h\|_{H^m}\| \Lambda^{m} h\|_{L^2}^2,
\end{aligned}
\eq
where $\tilde{\lambda}=\lambda \sigma^{-N} N^2$ and $C=C(\gamma, m, k, l)$ is a positive constant.
\end{lemma}
\begin{proof}For notational simplicity, let us write $g_k:=(h+\sigma)^{k(N-2)}-\sigma^{k(N-2)}$ in the proof. 

	We first show \eqref{gap} in the case $k=0$. In this case, we observe
	\bq\label{es0}
	\begin{aligned}
		\mathcal{R}_{m}(0)  
		&=\lambda \sigma^{-N} N \int \Lambda^{m-l} u \cdot \Lambda^{m-l}   \{ (h+\sigma)^{N-1} \nabla h \} \, dx \\
		&=\lambda \sigma^{-N} N \int \Lambda^{m-l} u \cdot [\Lambda^{m-l},\, (h+\sigma)^{N-2}] \{ (h+\sigma)\nabla h\}   \, dx + \tilde{\lambda} \mathcal{P}_{m}(1)\\
		&=\frac{\lambda \sigma^{-N} N}{2} \int \Lambda^{m-l} u \cdot [\Lambda^{m-l},\, (h+\sigma)^{N-2}]  \nabla (h+\sigma)^2   \, dx + \tilde{\lambda} \mathcal{P}_{m}(1).
	\end{aligned}
	\eq 
On the other hand, the following identity holds:
	\bq\nonumber
	[\Lambda^{m-l}, \, (h+\sigma)^{N-2}]=[\Lambda^{m-l}, \, (h+\sigma)^{N-2}-\sigma^{N-2}] = [\Lambda^{m-l}, \, g_1]
	\eq 
	in the operator sense with a given constant $\sigma$. Thus, by using \eqref{comm0} with sufficiently small $\varepsilon>0$, we obtain
	\bq\label{es1}
	\begin{aligned}
		\| [\Lambda^{m-l} , g_1] \, \nabla (h+\sigma)^2\|_{L^2} 
		&\leq C \|g_1\|_{H^{\frac{d}{2}+1+\varepsilon}}\|\Lambda^{m-l}(h+\sigma)^2 \|_{L^2}
		+ C\| \nabla (h+\sigma)^2\|_{H^{\frac{d}{2}+\varepsilon}} \|\Lambda^{m-l} g_1 \|_{L^2} \\
		&\leq C \|g_1\|_{H^m}\|\Lambda^{m}(h+\sigma)^2 \|_{L^2}
		+ C\| \nabla (h+\sigma)^2\|_{H^{m-1}} \|\Lambda^{m} g_1 \|_{L^2}
		\\
		&\leq C \|h\|_{H^m}\|\Lambda^{m}h \|_{L^2}
	\end{aligned}
	\eq 
 for some $C = C(\gamma, m, l) > 0$. Here we used \eqref{chain rule} to get
	\bq\nonumber
	\begin{aligned}
		&\|\Lambda^{m} (h+\sigma)^2\|_{L^2}  \leq C  \|\Lambda^m h\|_{L^2} , \\
		&\|\nabla (h+\sigma)^2\|_{H^{m-1}} \leq C  \|\nabla h\|_{H^{m-1}} \leq C \| \Lambda^m h\|_{L^2}, \quad \mbox{and}\\
		 &\|\Lambda^m g_1\|_{L^2} \leq \|g_1\|_{H^m} \leq C \|h\|_{H^m}.
	\end{aligned}
	\eq 
	Thus, we combine \eqref{es0} with \eqref{es1} to deduce that
	\[
	\begin{aligned}
		\lt|  \mathcal{R}_{m}(0) - \tilde{\lambda} \mathcal{P}_{m}(1) \rt|
		&\leq C \lambda \| \Lambda^{m-l} u\|_{L^2} \| [\Lambda^{m-l} ,\, (h+\sigma)^{N-2}-\sigma^{N-2}] \nabla (h+\sigma)^2\|_{L^2} \\
		&\leq C \lambda \| \Lambda^{m} u\|_{L^2}  \|h\|_{H^m} \|\Lambda^{m}h\|_{L^2}\\
		&\leq C\lambda \|h\|_{H^m}\| \Lambda^{m} u\|_{L^2}^2+  C\lambda  \|h\|_{H^m}\| \Lambda^{m} h\|_{L^2}^2
	\end{aligned}
	\]
	by Young's inequality with a positive constant $C=C(\gamma, m, l)$.
	
	Meanwhile, the case $k\ge 1$ requires a more delicate analysis since the operator $\Lambda^{-l}$ and $(h+\sigma)^{k(N-2)}$ do not commute, i.e. $[\Lambda^{-l},\, (h+\sigma)^{k(N-2)}]\neq 0$ for $k\ge 1$.   
	Thus we use the integration by parts twice as follows:
	\begin{align*}
		\mathcal{R}_m(k)
		& = -\lambda \sigma^{-N}   \int  \lt([\Lambda^{-l} \nabla \cdot, \, (h+\sigma)^{k(N-2)}]  \Lambda^{m-kl} u\rt)  \Lambda^{m-(k+1)l} \{ (h+\sigma)^{N}-\sigma^{N}\}\, dx\\
		&\quad + \lambda \sigma^{-N} \int \lt(\nabla(h+\sigma)^{k(N-2)}   \cdot \Lambda^{m-(k+1)l} u\rt)  \Lambda^{m-(k+1)l} \{ (h+\sigma)^{N}-\sigma^{N}\}\, dx\\
		&\quad + \lambda \sigma^{-N} \int \lt(  (h+\sigma)^{k(N-2)}   \Lambda^{m-(k+1)l}  u\rt) \Lambda^{m-(k+1)l} \nabla\{ (h+\sigma)^{N}-\sigma^{N}\}\, dx\\
		&=: \sfJ_{1} +\sfJ_{2} +\sfJ_{3}.
	\end{align*}
Then we use Lemma \ref{lem_moser} and arguments in Remark \ref{rem2} to obtain, for sufficiently small $\varepsilon <m-d/2-1$,
\[
	\begin{aligned}
		\| [\Lambda^{-l} \nabla \cdot, \, g_k]  \Lambda^{m-kl} u\|_{L^2}
		&\leq C \|g_k\|_{H^{\frac{d}{2}+1+\varepsilon}} \|\Lambda^{m-(k+1)l} u\|_{L^2} +	C\|g_k\|_{H^{\frac{d}{2}+1-l+\varepsilon}} \|\Lambda^{m-kl} u\|_{L^2})\\
		&\leq C \|g_k\|_{H^{m}} \|\Lambda^{m} u\|_{L^2}\\
		&\leq C\lt(\|g_k\|_{L^2} + \|\Lambda^m g_k\|_{L^2}\rt) \|\Lambda^m u\|_{L^2}\\
		&\leq C\|\Lambda^m h\|_{L^2} \|\Lambda^m u\|_{L^2},
	\end{aligned}
\]
where $C=C(\gamma, m, d,\alpha, k)$ is a positivie constant. Thus, we get
\[
	\begin{aligned}
		\sfJ_{1} 
		& \leq C \lambda \|  [\Lambda^{-l} \nabla \cdot, \, g_k]\,  \Lambda^{m-kl} u\|_{L^2} \|\Lambda^{m-(k+1)l} \{ (h+\sigma)^{N}-\sigma^{N}\} \|_{L^2} \\
		& \leq C \lambda  \|\Lambda^m h\|_{L^2} \|\Lambda^{m} u\|_{L^2}  \|\Lambda^{m} \{ (h+\sigma)^{N}-\sigma^{N}\} \|_{L^2} \\
		& \leq  C \lambda \|h\|_{H^{m}} \|\Lambda^{m} u\|_{L^2}^2+ C \lambda \|h\|_{H^{m}}  \|\Lambda^{m} h\|_{L^2}^2.  
	\end{aligned}
\]
For $\sfJ_2$, one has
\[
	\begin{aligned}
		\sfJ_{2} 
		&\leq C \lambda \|\nabla(h+\sigma)^{k(N-2)}\|_{L^{\infty}}  \|\Lambda^{m-(k+1)l} u\|_{L^2} \| \Lambda^{m-(k+1)l} \{ (h+\sigma)^{N}-\sigma^{N}\}\|_{L^2}\\
		&\leq C \lambda \|\nabla h\|_{L^{\infty}} \| \Lambda^{m} u \|_{L^2}\| \Lambda^{m}  \{ (h+\sigma)^{N}-\sigma^{N}\} \|_{L^2}\\
		&\leq C \lambda \|\nabla h\|_{L^{\infty}} \| \Lambda^{m} u \|_{L^2}\| \Lambda^{m} h \|_{L^2}\\
		&\leq C \lambda \|\nabla h\|_{L^{\infty}} \| \Lambda^{m} h \|_{L^2}^2 + C \lambda \|\nabla h\|_{L^{\infty}}\| \Lambda^{m} u \|_{L^2}^2
	\end{aligned}
\]
for some $C=C(\gamma, m, d,\alpha, k) > 0$.
	
We next rewrite $\sfJ_3$ as 
\[\begin{aligned}
		\sfJ_{3} 
		&=  \lambda \sigma^{-N} N \int  (h+\sigma)^{k(N-2)}  \Lambda^{m-(k+1)l} u \cdot \Lambda^{m-(k+1)l} \{ (h+\sigma)^{N-1}\nabla h \}\, dx\\
		&=  \lambda \sigma^{-N}N\int  (h+\sigma)^{k(N-2)}  \Lambda^{m-(k+1)l} u \cdot [\Lambda^{m-(k+1)l}, \, (h+\sigma)^{N-2}]\{(h+\sigma)\nabla h \}\, dx + \tilde{\lambda} \mathcal{P}_m(k+1)\\
		&=   \frac{ \lambda \sigma^{-N}N }{2} \int  (h+\sigma)^{k(N-2)}  \Lambda^{m-(k+1)l} u \cdot [\Lambda^{m-(k+1)l}, \, g_1]\nabla (h+\sigma)^2 \, dx + \tilde{\lambda} \mathcal{P}_m(k+1).
\end{aligned}\]
Similarly as in \eqref{es1}, we have 
	\bq\nonumber
	\| [\Lambda^{m-(k+1)l} , g_1]\, \nabla (h+\sigma)^2 \|_{L^2} 
	\leq C \|h\|_{H^m} \|\Lambda^{m}h \|_{L^2}
	\eq 
	for $k< m/l -1$ and some $C=C(\gamma, m, k, d,\alpha) > 0$.
Thus, it implies
\[
	\begin{aligned}
		|\sfJ_{3} - \tilde{\lambda} \mathcal{P}_m(k+1)|
  \leq C \lambda  \|h\|_{H^m}\|\Lambda^{m} u\|_{L^2}\|\Lambda^{m}h \|_{L^2}  \leq C \lambda \| h\|_{H^m} \|\Lambda^{m} u\|_{L^2}^2+C \lambda \| h\|_{H^m} \|\Lambda^{m}  h\|_{L^2}^2.
	\end{aligned}
\]
Hence, we collect all the estimates for $\sfJ_i$'s to yield 
	\begin{align*}
		\lt|\mathcal{R}_m (k) - \tilde{\lambda} \mathcal{P}_m (k+1) \rt|
		&\leq  C \lambda \| h\|_{H^m} \| \Lambda^{m} h \|_{L^2}^2 + C \lambda \| h\|_{H^m} \| \Lambda^{m} u \|_{L^2}^2 
	\end{align*}
	as desired.
\end{proof}

Now, we are ready to present the proof of Proposition \ref{higher_u}.
\begin{proof}[Proof of Proposition \ref{higher_u}]
Let $m>1+d/2$. Through the integration by parts, we find
\bq\nonumber
\begin{aligned}
	&\frac{1}{2} \frac{d}{dt} \left( \int  |\Lambda^m h|^2 \, dx  +\int  |\Lambda^m u|^2 \, dx  \right)+\nu \int  |\Lambda^m u|^2 \, dx \\
	&\qquad=  -\int \Lambda^m h\, [\Lambda^m, u\cdot ]\, \nabla h \, dx
	-\int \Lambda^m u \cdot ([\Lambda^m, u\cdot ]\, \nabla u )\, dx-\frac{1}{2} \int  (\nabla \cdot u) (|\Lambda^m u|^2+|\Lambda^m h|^2)\, dx  \\
	&\qquad\quad -\frac{1}{N} \int \Lambda^m h\, \Lambda^m \{ (h+\sigma) (\nabla \cdot u)\} \, dx
	-\frac{1}{N} \int \Lambda^m u \cdot \Lambda^m \{ (h+\sigma) \nabla h\} \, dx \\
	& \qquad\quad + \lambda \sigma^{-N} \int \Lambda^m u \cdot \Lambda^m \nabla \Lambda^{\alpha-d} \{ (h+\sigma)^{N} -\sigma^{N}\} \, dx \nonumber\\
	&\qquad =: \sum_{i=1}^5 \sfK_i  + \mathcal{R}_m(0) ,
\end{aligned}
\eq
where $\mathcal{R}_m(0)$ is given as in \eqref{R_k} for the case $k=0$.

We use the commutator estimate \eqref{comm0} with taking $\varepsilon>0$ sufficiently small so that 
\bq\nonumber
\begin{aligned}
\sfK_1 &\leq C \| \Lambda^m h\|_{L^2} ( \| u\|_{H^m} \|\Lambda^{m} h\|_{L^2} + \| h\|_{H^m} \|\Lambda^m u\|_{L^2} )\\
&\leq C (\| u\|_{H^m}+ \|h\|_{H^m})  \| \Lambda^m h\|_{L^2}^2 +  C\| h\|_{H^m} \|\Lambda^m u\|_{L^2}^2
\end{aligned}
\eq
and
\[
\sfK_2 \leq C  \| u\|_{H^m} \| \Lambda^m u\|_{L^2}^2
\]
along with Young's inequality.
Here, $C>0$ depends on $m$.
It is easy to get
\begin{align*}
\sfK_3& \leq \frac{1}{2} \|\nabla \cdot u\|_{L^{\infty}} \|\Lambda^m u\|_{L^2}^2+\frac{1}{2} \|\nabla \cdot u\|_{L^{\infty}}\|\Lambda^m h\|_{L^2}^2.
\end{align*}
We use the integration by parts twice so that
\bq\nonumber
\begin{aligned}
\sfK_4 + \sfK_5=&-\frac{1}{N} \int \Lambda^m h \, \Lambda^m ( h (\nabla \cdot u)) \, dx
-\frac{1}{N} \int \Lambda^m u \cdot \Lambda^m ( h \nabla h) \, dx \\
=&-\frac{1}{N} \int \Lambda^m h\,  [\Lambda^m, h] (\nabla \cdot u) \, dx 
-\frac{1}{N} \int \Lambda^m u \cdot [\Lambda^m, h]\nabla h \, dx 
+\frac{1}{N} \int \Lambda^m h \nabla h  \cdot \Lambda^m u\, dx. 
\end{aligned}
\eq 
Then it follows from \eqref{comm0} with $\varepsilon>0$ small enough that
\bq\nonumber
\begin{aligned}
\sfK_4+\sfK_5
	&\leq C  (\|h\|_{H^m}+\|\nabla h\|_{L^{\infty}} ) \| \Lambda^m u\|_{L^2} \| \Lambda^m h\|_{L^2} +C \|u\|_{H^m}\| \Lambda^m h\|_{L^2}^2  \\  
	&\leq C  (\|h\|_{H^m}+\|\nabla h\|_{L^{\infty}} + \|u\|_{H^m} ) \| \Lambda^m h\|_{L^2}^2
	 +C  (\|h\|_{H^m}+\|\nabla h\|_{L^{\infty}} ) \| \Lambda^m u\|_{L^2}^2 
\end{aligned}
\eq 
thanks to the Young's inequality. Here, a positive constant $C$ depends on $\gamma$ and $m$. Thus, we obtain 
\bq\label{s2}
\begin{aligned}
&\frac{1}{2} \frac{d}{dt} \left(\int  |\Lambda^m h|^2 \, dx + \int  |\Lambda^m u|^2 \, dx \right) +\nu \int  |\Lambda^m u|^2 \, dx \\ 
&\quad  \leq  C(\| u\|_{H^m}+  \|h\|_{H^m} )  \|\Lambda^m  h\|_{L^2}^2
+ C (\| u\|_{H^m}+\|h\|_{H^m} ) \| \Lambda^m u\|_{L^2}^2 + \mathcal{R}_{m} (0)  
\end{aligned}
\eq 
for some $C=C(\gamma, m) > 0$.

Now we need to control the term $\mathcal{R}_m(0)$ from the Riesz interaction.

\noindent $\bullet$ (Sub-Manev potential case: $-2\leq \alpha-d \leq -1$.) Here, we use Lemmas \ref{Sobolev_T} and \ref{tool} to get
\bq\label{ellp_good}
\begin{aligned}
\mathcal{R}_m(0) 
&  \leq \lambda\sigma^{-N} \| \Lambda^m u \|_{L^2} \|\Lambda^{m+\alpha-d} \nabla  \{(h+\sigma)^N -\sigma^N \}\|_{L^2}   \\
& \leq \lambda \sigma^{-N}\| \Lambda^m u \|_{L^2} \| \Lambda^{m-1} \nabla \{(h+\sigma)^N -\sigma^N \}\|_{L^2} \\
&\leq 2C_0 \lambda \| \Lambda^{m} u \|_{L^2} \| \Lambda^m   h \|_{L^2}  + C\lambda\|h\|_{H^m}\|\Lambda^m h\|_{L^2}\|\Lambda^m u\|_{L^2}(1+\|h\|_{H^m})^{m-2}\\
&\leq C_0 \lambda \| \Lambda^m u \|_{L^2}^2 +C_0 \lambda   \| \Lambda^{m-1} \nabla h \|_{L^2}^2 + C\lambda\|h\|_{H^m}(1+\|h\|_{H^m})^{m-2}\lt(\|\Lambda^m u\|_{L^2}^2 + \|\Lambda^m h\|_{L^2}^2 \rt),
\end{aligned}
\eq
where $C_0$ is given as
\[
C_0 = C_d (2\sigma)^{-1} N \max \{ 2^{-(N-1)}, 2^{N-1}\}. 
\]

\noindent $\bullet$ (Super-Manev potential case: $-1< \alpha-d <0$.) In this case, it is difficult to directly estimate the term $\mathcal{R}_m(0)$ as we just previously did. 
Instead, we shall use an iteration argument, which uses Lemmas \ref{Gap} and  \ref{higher_it} repeatedly, to reduce the order of derivative on the Riesz interaction term. Then, we arrive at 
\bq
\begin{aligned}\label{R_bad}
\mathcal{R}_{m}(0) 
&\leq  C \lambda ( \|u\|_{H^m}+\|h\|_{H^m} + \|u\|_{H^m} \|h\|_{H^m}) \|\Lambda^{m} u\|_{L^2}^2 + C \lambda \|h\|_{H^m} \|\Lambda^{m} h\|_{L^2}^2  \\ 
&\quad - \frac{\tilde{\lambda}}{2} \frac{d}{dt} \int (h+\sigma)^{k(N-2)} |\Lambda^{m-kl} u|^2 \, dx  + \tilde{\lambda}\mathcal{R}_m(1) \\
&\leq 
C \sum_{k=1}^2 \lambda^k (  \|u\|_{H^m} + \|h\|_{H^m}+\|u\|_{H^m}\|h\|_{H^m}) \|\Lambda^{m} u\|_{L^2}^2  +C \sum_{k=1}^2 \lambda^k  \|h\|_{H^m}  \|\Lambda^{m} h\|_{L^2}^2 \\
& \quad - \sum_{k=1}^2\frac{\tilde{\lambda}^k}{2} \frac{d}{dt} \int (h+\sigma)^{k(N-2)} |\Lambda^{m-kl} u|^2 \, dx 
+ \tilde{\lambda}^2\mathcal{R}_m(2) \\
&\leq \ \cdots \\
&\leq 
C \sum_{k=1}^{k_0} \lambda^k (  \|u\|_{H^m} + \|h\|_{H^m}+\|u\|_{H^m}\|h\|_{H^m}) \|\Lambda^{m} u\|_{L^2}^2  +C \sum_{k=1}^{k_0} \lambda^k  \|h\|_{H^m}  \|\Lambda^{m} h\|_{L^2}^2 \\
& \quad - \sum_{k=1}^{k_0}\frac{\tilde{\lambda}^k}{2} \frac{d}{dt} \int (h+\sigma)^{k(N-2)} |\Lambda^{m-kl} u|^2 \, dx 
+ \tilde{\lambda}^{k_0}\mathcal{R}_m(k_0) \\
\end{aligned}
\eq 
for $1\leq k_0 <m/l$ and recall that $\tilde{\lambda}=\lambda \sigma^{-N}N^2$.
We now take $k_0\in \mathbb{N}$ as a minimum value for $k$ satisfying 
\bq\nonumber
\max\lt\{1,\,  \frac{1}{l} -2 \rt\} \leq k <\frac{m}{l} 
\eq 
which is valid when $0<l<1/2$. Then by applying Lemmas \ref{Sobolev_T} and \ref{tool}, one obtains
\bq\label{R2}
\begin{aligned}
\mathcal{R}_m(k_0)
&= \lambda \sigma^{-N} \int (h+\sigma)^{k_0(N-2)} \Lambda^{m-k_0l}u\cdot \Lambda^{m-1}\nabla \Lambda^{1-(k_0+2)l} \{(h+\sigma)^N -\sigma^N\} \, dx \\
&\leq  \lambda \sigma^{k_0 (N-2)-N}  \max\{ 2^{-k_0 (N-2)}, 2^{k_0 (N-2)}\}  \| \Lambda^{m-k_0l} u\|_{L^2} \| \Lambda^{m-1}\nabla  \{ (h+\sigma)^{N}-\sigma^{N}\}\|_{L^2} \\
& \leq  \widetilde{C}_{0} \lambda     \| \Lambda^{m} u \|_{L^2} \| \Lambda^{m-1} \nabla h \|_{L^2} + C \lambda \|h\|_{H^m}\|\Lambda^m h\|_{L^2}\|\Lambda^m u\|_{L^2}(1+\|h\|_{H^m})^{m-2}\\
&\leq \frac{\widetilde{C}_{0}\lambda}{2}\| \Lambda^{m} u \|_{L^2}^2 +\frac{ \widetilde{C}_{0}\lambda}{2} \|\Lambda^{m-1} \nabla h \|_{L^2}^2 + C \lambda \|h\|_{H^m}(1+\|h\|_{H^m})^{m-2}\lt(\|\Lambda^m u\|_{L^2}^2 + \|\Lambda^m h\|_{L^2}^2 \rt),
\end{aligned}
\eq 
where $\widetilde{C}_0$ is written as
\[
\widetilde{C}_{0} := C_d \sigma^{k_0 (N-2)-1}N  \max\{ 2^{-k_0 (N-2)-(N-1)}, 2^{-k_0 (N-2)+N-1},  2^{k_0 (N-2)+N-1}\}.
\]
Thus, by \eqref{R_bad} and \eqref{R2}, we obtain 
\bq\nonumber
\begin{aligned}
\mathcal{R}_m(0) &\leq C \sum_{k=1}^{k_0} \lambda^k (  \|u\|_{H^m} + \|h\|_{H^m}+\|u\|_{H^m}\|h\|_{H^m}) \|\Lambda^{m} u\|_{L^2}^2  +C \sum_{k=1}^{k_0} \lambda^k  \|h\|_{H^m}  \|\Lambda^{m} h\|_{L^2}^2 \\
&\quad + C_{1} \lambda^{k_0+1}  \| \Lambda^{m} u \|_{L^2}^2 + C_{1}\lambda^{k_0+1} \|\Lambda^{m} h \|_{L^2}^2- \sum_{k=1}^{k_0}\frac{\tilde{\lambda}^k}{2} \frac{d}{dt} \int (h+\sigma)^{k(N-2)} |\Lambda^{m-kl} u|^2 \, dx \\
&\quad + C \lambda^{k_0+1}\|h \|_{H^m}(1+\|h\|_{H^m})^{m-2}(\| \Lambda^{m} u \|_{L^2}^2 + \|\Lambda^{m} h \|_{L^2}^2),
\end{aligned}
\eq
where $C =C (\gamma, m, \alpha, d)$ is a positive constant and $C_1 > 0$ is given as
\bq\nonumber
C_{1} :=(\sigma^{-N}N^2)^{k_0}2^{-1}\widetilde{C}_{0}\\
 = C_d  \sigma^{-(2k_0 +1)} N^{2k_0+1}  \max\{ 2^{-k_0 (N-2)-N}, 2^{-(k_0-1) (N-2)},  2^{(k_0+1) (N-2)}\}. 
\eq 
This yields
\bq\label{ellp_bad}
\begin{aligned}
\mathcal{R}_m(0) &\leq  \lt( C\hat\lambda (\|u\|_{H^m} + \|h\|_{H^m} + \|u\|_{H^m}\|h\|_{H^m})+ C_{1}\lambda^{k_0 + 1} +C \lambda^{k_0+1}\|h \|_{H^m}(1+\|h\|_{H^m})^{m-2} \rt) \|\Lambda^{m} u\|_{L^2}^2\\
 &\quad + \lt( C(\hat\lambda + \lambda^{k_0+1}(1+\|h\|_{H^m})^{m-2} ) \|h\|_{H^m} + C_{1}\lambda^{k_0 + 1} \rt) \|\Lambda^{m} h\|_{L^2}^2 \\
  &\quad - \sum_{k=1}^{k_0}\frac{\tilde{\lambda}^k}{2} \frac{d}{dt} \int (h+\sigma)^{k(N-2)} |\Lambda^{m-kl} u|^2 \, dx.
\end{aligned}
\eq 
Here $\hat\lambda  = \sum_{k=1}^{k_0} \lambda^k > 0$.

Hence, we combine \eqref{s2} with \eqref{ellp_good} for the case $-2\leq \alpha-d \leq -1$ and with \eqref{ellp_bad} for the case $-1<\alpha-d <0$ so that we conclude \eqref{h1} and \eqref{h2}, respectively. 
\end{proof}

We close this section by giving the dissipation estimate for $\nabla h$ in the $H^{m-1}$ norm. The proof is based on the hypocoercivity-type estimate. Precisely, we consider the time-derivative of the following integral 
\[
\int \nabla (\Lambda^s h) \cdot \Lambda^s u\, dx
\]
with $s \leq m-1$ which gives an appropriate dissipation rate for $h$.  Note that for $\mu > 0$ small enough 
\[
\|(h,u)\|_{H^m}^2  + \mu\sum_{s=0}^{m-1}\int \nabla (\Lambda^s h) \cdot \Lambda^s u\, dx \sim \|(h,u)\|_{H^m}^2. 
\]

\begin{lemma}\label{higher_h}
Let the assumptions of Lemma \ref{zero_u} be satisfied. Then for $0\leq s\leq m-1$ we have
\[
	\begin{aligned}
		&\frac{d}{dt} \int \nabla (\Lambda^s h) \cdot \Lambda^s u\, dx
		+\nu \int \nabla (\Lambda^s h) \cdot \Lambda^s u\, dx 
		+\frac{\sigma}{N} \int |\nabla (\Lambda^s h) |^2\, dx\\
		&\quad \leq \lt(C\|u\|_{H^m} +C\|h\|_{H^m} +\frac{\sigma}{N}\rt)  \|\Lambda^{s+1} u\|_{L^2}^2   \\
		&\qquad+( C\|u\|_{H^m} + C\|h\|_{H^m}  +2C_0\lambda +C \lambda \|h\|_{H^m}(1+\|h\|_{H^m})^{m-2}) \|\nabla \Lambda^s h\|_{L^2}^2
	\end{aligned}
\]
 for some $C=C(\gamma, d, s, m) > 0$, where $C_0$ is appeared in Proposition \ref{higher_u}.
	\end{lemma}
	\begin{proof}
	Let $0\leq s\leq m-1$.  We integrate by parts to see that
	\bq\nonumber
	\begin{aligned}
		\frac{d}{dt} \int \nabla (\Lambda^s h) \cdot \Lambda^s u\, dx
&=\int \Lambda^s (u \cdot \nabla h) \Lambda^s(\nabla \cdot u)\, dx +\frac{1}{N} \int \Lambda^s \{(h+\sigma) (\nabla \cdot u) \} \Lambda^s(\nabla \cdot u)\, dx \\
& \quad 
-\int \nabla (\Lambda^s h) \cdot \Lambda^s(u\cdot \nabla u)\, dx 
-\frac{1}{N} \int \nabla (\Lambda^s h) \cdot \Lambda^s \{ (h+\sigma)\nabla h\}\, dx \\
& \quad 
-\nu \int \nabla (\Lambda^s h) \cdot \Lambda^s u \, dx
+\sigma^{-N} \lambda \int  \nabla  (\Lambda^s h) \cdot \Lambda^s \nabla \Lambda^{\alpha-d} \{ (h+\sigma)^N-\sigma^N\}\, dx\\
&=:  \sum_{i=1}^5 \sfL_i+\widetilde{R}_{s},
	\end{aligned}
	\eq 
	where we use Lemma \ref{lem_moser} to estimate the terms $\sfL_i$ ($1 \le i \le 4$) as follows:
	\[
	\begin{aligned}
		\sfL_1 
		&= \int \Lambda^s (u \cdot \nabla h) \Lambda^s(\nabla \cdot u)\, dx\\
		& \leq C (\|u\|_{L^{\infty}} \|\Lambda^{s} \nabla h\|_{L^2} + \|\nabla h\|_{L^{\infty}}\|\Lambda^s u\|_{L^2})\|\Lambda^{s}(\nabla \cdot u) \|_{L^2}\\
		&\leq C \|u\|_{L^{\infty}} \|\Lambda^{s} \nabla h\|_{L^2} \|\Lambda^{s+1}u \|_{L^2}+C \|\nabla h\|_{L^{\infty}}\|\Lambda^{s+1} u\|_{L^2}^2,\\
		\sfL_2
		&=\frac{1}{N} \int \Lambda^s (h \nabla \cdot u)  \Lambda^s(\nabla \cdot u)\, dx +\frac{\sigma}{N} \int | \Lambda^s (\nabla \cdot u) |^2\, dx \\
		& \leq C(\|h\|_{L^{\infty}} \|\Lambda^{s} (\nabla \cdot u)\|_{L^2} + \|\nabla \cdot u\|_{L^{\infty}}\|\Lambda^s h\|_{L^2}) \|\Lambda^{s}( \nabla \cdot u)\|_{L^2} + \frac{\sigma}{N} \|\Lambda^{s}(\nabla \cdot u) \|_{L^2}^2\\
		&\leq  C \|\nabla \cdot u\|_{L^{\infty}}\|\Lambda^s \nabla h\|_{L^2} \|\Lambda^{s+1} u\|_{L^2} + C\|h\|_{L^{\infty}} \|\Lambda^{s+1} u\|_{L^2}^2  + \frac{\sigma}{N} \|\Lambda^{s+1} u \|_{L^2}^2,\\
		\sfL_3 
		&=-\int \nabla  (\Lambda^s h) \cdot \Lambda^s(u\cdot \nabla u)\, dx  \\
		&\leq C (\|u\|_{L^{\infty}} \|\Lambda^{s+1} u\|_{L^2} + \|\nabla u\|_{L^{\infty}}\|\Lambda^s u\|_{L^2})\|\nabla \Lambda^{s} h\|_{L^2} \\
		&\leq C (\|u\|_{L^{\infty}} + \|\nabla u\|_{L^{\infty}}) \|\Lambda^{s+1} u\|_{L^2}\|\nabla \Lambda^{s} h\|_{L^2} ,
			\end{aligned}
	\]
	and
		\[
	\begin{aligned}
		\sfL_4 
		&= -\frac{1}{N} \int \nabla (\Lambda^s h) \cdot \Lambda^s ( h \nabla h)\, dx  -\frac{\sigma}{N} \int |\nabla (\Lambda^s h)|^2\, dx\\
		&\leq C \| \nabla \Lambda^{s} h\|_{L^2} ( \|h\|_{L^{\infty}}\|\nabla \Lambda^{s} h\|_{L^2}+ \|\nabla h\|_{L^{\infty}} \|\Lambda^{s} h\|_{L^2}) -\frac{\sigma}{N} \int | \nabla (\Lambda^s h) |^2\, dx\\
		&\leq C(  \|h\|_{L^{\infty}} + \|\nabla h\|_{L^{\infty}}  ) \| \nabla \Lambda^{s} h\|_{L^2}-\frac{\sigma}{N} \int | \nabla (\Lambda^s h) |^2\, dx
	\end{aligned}
	\]
	for some $C=C(s,\gamma) > 0$. 
	
	Thus, we collect the estimates for $\sfL_i$'s to get
	\bq\nonumber
	\begin{aligned}
		&\frac{d}{dt} \int \nabla (\Lambda^s h) \cdot \Lambda^s u\, dx
		+\nu \int \nabla (\Lambda^s h) \cdot \Lambda^s u\, dx 
		+\frac{\sigma}{N} \int |\nabla (\Lambda^s h) |^2\, dx\\
		&\quad \leq (C\|u\|_{L^{\infty}} + C\|\nabla \cdot u\|_{L^{\infty}} +C\|h\|_{L^{\infty}} + C\|\nabla h\|_{L^{\infty}} +\frac{\sigma}{N})  \|\Lambda^{s+1} u\|_{L^2}^2  \\
		&\qquad  +(C \|u\|_{L^{\infty}} + C \|\nabla \cdot u\|_{L^{\infty}} + C \|h\|_{L^{\infty}}+C\|\nabla h\|_{L^{\infty}}) \|\nabla \Lambda^s h\|_{L^2}^2  +\widetilde{R}_s.
	\end{aligned}
	\eq
	By using the neutrality condition \eqref{neu_con2}, Lemmas \ref{Sobolev_T} and \ref{tool}, we next estimate 
	\bq\nonumber
	\begin{aligned}
		\widetilde{R}_{s}
		&= \lambda \sigma^{-N} \int  \nabla(\Lambda^s h) \cdot \Lambda^s \nabla \Lambda^{\alpha-d} \{ (h+\sigma)^N-\sigma^N\}\, dx \\
		&\leq \lambda \sigma^{-N}  \|\Lambda^{s} \nabla h \|_{L^2} \|\Lambda^{s+\alpha-d} \nabla  \{ (h+\sigma)^N-\sigma^N\} \|_{L^2}\\
		&\leq \sigma^{-N} \lambda \|\Lambda^s \nabla h \|_{L^2} \|\Lambda^s \nabla \{ (h+\sigma)^N-\sigma^N\} \|_{L^2}\\
		&\leq 2C_0 \lambda \|\Lambda^s \nabla h \|_{L^2}^2 + C \lambda \|h\|_{H^m}(1+\|h\|_{H^m})^{m-2}\|\Lambda^s \nabla h \|_{L^2}^2. 
	\end{aligned}
	\eq 
Hence, we have
	\bq\nonumber
	\begin{aligned}
	&\frac{d}{dt} \int \nabla (\Lambda^s h) \cdot \Lambda^s u\, dx +\nu \int \nabla (\Lambda^s h) \cdot \Lambda^s u\, dx +\frac{\sigma}{N} \int |\nabla (\Lambda^s h) |^2\, dx\\
&\quad \leq (C\|u\|_{L^{\infty}} + C\|\nabla \cdot u\|_{L^{\infty}} +C\|h\|_{L^{\infty}} + C\|\nabla h\|_{L^{\infty}} +\frac{\sigma}{N})  \|\Lambda^{s+1} u\|_{L^2}^2  \\
&\qquad  +(C \|u\|_{L^{\infty}} + C \|\nabla \cdot u\|_{L^{\infty}} + C \|h\|_{L^{\infty}}+C\|\nabla h\|_{L^{\infty}} +2C_0\lambda+ C \lambda \|h\|_{H^m}(1+\|h\|_{H^m})^{m-2}) \|\nabla \Lambda^s h\|_{L^2}^2,  
	\end{aligned}
	\eq
	and applying the Sobolev inequality to the above yields the desired result. 
	\end{proof}

%
%
%
%
%

\section{Global existence of solutions: proof of Theorem \ref{main_thm1}}\label{sec:4}
\setcounter{equation}{0}

In this section, we present the global-in-time existence and uniqueness of the regular solutions $(h, u)$ to \eqref{np_ER2}-\eqref{ini_ER2} by extending the local-in-time solution to the global one based on the classical continuation argument. For this, we obtain the uniform-in-time bound estimates of the solution by using a priori estimates established in the previous section.
\begin{proposition}\label{prop}
	Let $T>0$, $m>1+d/2$, and $(h, u)\in \mathcal{C}([0, T); H^m(\T^d) \times H^m(\T^d))$ be a classical solution to \eqref{np_ER2}-\eqref{ini_ER2}. 
	Suppose that $\|u\|_{L^{\infty}([0, T];H^m)}+\|h\|_{L^{\infty}([0, T];H^m)} \leq \varepsilon_1^2 \ll 1$ so that 
\[
	\sup_{0\leq t\leq T}\|h(t)\|_{L^\infty} \leq \frac{\sigma}{2}.
\]
	Then there exists a positive constant $C^{\ast}$ independent of $T$ such that 
	\bq\nonumber
	\sup_{0\leq t\leq T}\|(h, u)(t)\|_{H^m}^2 \leq C^{\ast}\|(h_0, u_0)\|_{H^m}^2.
	\eq
\end{proposition}
\begin{proof}
Since the proof of the sub-Manev potential case ($-2\leq \alpha-d \leq -1$) is rather parallel to that of the super-Manev potential case ($0<\alpha-d <-1$), we only provide the details of the proof for the latter case. 

Note that
\[
(1+\|h\|_{H^m})^{m-2} \leq (1+\varepsilon_1)^{m-2} <C, 
\]
thus it follows from Proposition \ref{higher_u} that 
\bq\nonumber
\begin{aligned}
	&\frac{1}{2} \frac{d}{dt}\left(
	\sum_{k=0}^{k_0}\tilde{\lambda}^k \int  (h+\sigma)^{k(N-2)} |\Lambda^{m-kl} u|^2\, dx
	+ \int  |\Lambda^m h|^2 \, dx  \right) + \nu \|\Lambda^m u\|_{L^2}^2  \\
	&\quad \leq  \lt(C(\e_1 + \hat\lambda \e_1 + \lambda^{k_0+1} \e_1) + C_1  \lambda^{k_0 + 1}\rt)\|\Lambda^m u\|_{L^2}^2 +  \lt(C(\e_1 + \hat\lambda \e_1+ \lambda^{k_0+1}\e_1 ) + C_1 \lambda^{k_0 + 1}\rt)\|\Lambda^m h\|_{L^2}^2,
\end{aligned}
\eq
where $\hat\lambda$, $\tilde\lambda$, $k_0$, and $C_1$ are appeared in Proposition \ref{higher_u}.
With Lemma \ref{higher_h} multiplied by a constant $\mu/2>0$, we deduce that
\bq\label{highers}
\begin{aligned}
&\frac{1}{2} \frac{d}{dt}\left(
\sum_{k=0}^{k_0} \tilde{\lambda}^k \int  (h+\sigma)^{k(N-2)} |\Lambda^{m-kl} u|^2\, dx
+ \int  |\Lambda^m h|^2 \, dx +  \mu \int \nabla (\Lambda^{m-1} h) \cdot \Lambda^{m-1} u\, dx \right) \\
&\quad +\nu \|\Lambda^m u\|_{L^2}^2  +\frac{\mu \sigma }{2N}\|\Lambda^{m-1} \nabla h\|_{L^2}^2 +\frac{\mu \nu}{2}  \int \nabla (\Lambda^{m-1} h) \cdot \Lambda^{m-1} u\, dx\\
&\qquad \leq  \lt(C(\e_1(1+ \mu) + \hat\lambda \e_1 + \lambda^{k_0+1} \e_1) + C_1  \lambda^{k_0 + 1}  + \frac{\mu\sigma }{2N} \rt)\|\Lambda^m u\|_{L^2}^2 \cr
&\qquad \quad +  \lt(C(\e_1(1+ \mu) + (\hat\lambda + \lambda^{k_0+1})\e_1+\lambda \mu \e_1 ) + C_1 \lambda^{k_0 + 1} +  C_0 \mu \lambda \rt)\|\Lambda^m h\|_{L^2}^2.
\end{aligned}
\eq
 We also recall from Lemma \ref{zero_u} that
\bq\label{zero}
\begin{aligned}
\frac{1}{2}\frac{d}{dt} \left( \|u\|_{L^2}^2 + \|h\|_{L^2}^2 \right)
+\lt(\nu - C\varepsilon_1- C_0 \lambda\rt) \|u\|_{L^2}^2\leq  \lt(C \varepsilon_1+ C_0\lambda\rt)\|\nabla h\|_{L^2}^2.
\end{aligned}
\eq 
Then we combine \eqref{highers} with \eqref{zero} and use Young's inequality to yield
\bq\nonumber
\begin{aligned}
\frac{d}{dt}{\mathcal X}_m &\leq -(2\nu - C \varepsilon_1- 2C_0  \lambda) \|u\|_{L^2}^2\\
&\quad -\lt(2\nu -C \varepsilon_1(1+\mu) -C\e_1 \hat\lambda - C\e_1 \lambda^{k_0+1}  - 2C_1  \lambda^{k_0 + 1} - \mu\lt( \frac{\sigma }{N} + \frac{N \nu^2}{2\sigma}\rt) \rt) \|\Lambda^m u\|_{L^2}^2 \cr 
&\quad -\lt(\frac{ \mu \sigma}{2N}- C \varepsilon_1(1 + \mu) - C\e_1(\hat\lambda + \lambda^{k_0+1})-C\lambda \mu \e_1 - 2C_0 \lambda(1+\mu) - 2C_1 \lambda^{k_0 + 1}
 \rt) \| \Lambda^m h \|_{L^2}^2,
\end{aligned}
\eq
where
\[
{\mathcal X}_m = \sum_{k=0}^{k_0} \tilde{\lambda}^k \int  (h+\sigma)^{k(N-2)} |\Lambda^{m-kl} u|^2\, dx
+\|u\|_{L^2}^2
+\|\Lambda^m h\|_{L^2}^2
+\|h\|_{L^2}^2
+\mu \int \nabla (\Lambda^{m-1} h) \cdot \Lambda^{m-1} u\, dx.
\]
Here, together with Remark \ref{rem2}, we also note that
\[
\|u\|_{H^m} \sim \|u\|_{L^2} + \|\Lambda^m u\|_{L^2}.
\]
Now we first take $\mu >0$ small enough such that 
\[
2\nu >  \mu\lt( \frac{\sigma }{N} + \frac{N \nu^2}{2\sigma}\rt)
\]
and
\[
{\mathcal X}_m \sim \|(h,u)\|_{H^m}^2.
\]
We notice that we can choose $\lambda > 0$ sufficiently small compared to $\sigma$ so that $\tilde \lambda \ll 1$, $ C_0 \lambda \ll 1$, and $  C_1 \lambda \ll 1$. This, together with taking $\e_1$ small enough, yields that there exists $\xi > 0$ such that 
\[
\frac{d}{dt}{\mathcal X}_m + \xi {\mathcal X}_m \leq 0.
\]
Hence we have
\bq\label{time_est}
\|(h, u)(t)\|_{H^m}^2  \leq C^{\ast}e^{- \xi t} \|(h_0, u_0)\|_{H^m}^2 \quad \mbox{for all } t \in [0,T]
\eq
for some $C^{\ast}>0$ independent of $T$. This concludes the desired result.
\end{proof}

\begin{proof}[Proof of Theorem \ref{main_thm1}]
	We set 
	\bq\label{vare_1}
	\varepsilon_2^2 := \frac{\varepsilon_1^2}{2(1+C^{\ast})}
	\eq
	for $\varepsilon_1 > 0$ and $C^{\ast} > 0$ given in Proposition \ref{prop}.
	By Theorem \ref{local}, the following set is nonempty:
	\bq\nonumber
	\mathcal{S}:=\lt\{ T>0\, :\, \sup_{0\leq t\leq T}\|(h, u)(t)\|_{H^m}^2\leq \varepsilon_1^2\rt\},
	\eq  
	and assume for a contradiction that $\sup\mathcal{S}=: T^{\ast}<\infty$. 
	Then by Proposition \ref{prop} and \eqref{vare_1}, we have
	\bq\nonumber
	\varepsilon_1^2 = \sup_{0\leq t\leq T^{\ast}}\|(h, u)(t)\|_{H^m}^2 \leq C^{\ast} \|(h_0, u_0)\|_{H^m}^2  \leq C^{\ast} \varepsilon_2^2=\frac{C^{\ast}\varepsilon_1^2}{2(1+C^{\ast})}  <\varepsilon_1^2
	\eq 
which is a contradiction. Hence, the solution exists in $\mathcal{C}(\mathbb{R}_{+}; H^m(\T^d) \times H^m(\T^d))$ and \eqref{largetime} directly follows from \eqref{time_est}.
\end{proof}

%
%
%
%
\section{A priori estimate of large-time behavior: proof of Theorem \ref{main_thm2}}\label{sec:5}
\setcounter{equation}{0}
In this section, we investigate the a priori large-time behavior for the $L^2$-norm of the sufficiently regular solution $(h, u)$ to the system \eqref{np_ER2}-\eqref{ini_ER2} without any smallness assumptions on the initial data $(h_0, u_0)$.

To show the temporal decay estimate \eqref{largetime3}, we first introduce a function $\mathcal{L}=\mathcal{L}(t)$ by
\bq\nonumber
\mathcal{L} := \int (h+\sigma)^N  |u-m_c|^2\, dx
+ \int |(h+\sigma)^N-\sigma^N|^2\, dx 
+\frac{1}{2}|m_c|^2,
\eq
where 
\bq\nonumber
m_c(t) =  \sigma^{-N} \int (h(t)+\sigma)^N u(t)\, dx.
\eq 
We recall that $m_c$ decays to zero exponentially fast as time goes to infinity:
\bq\label{exp_mc}
m_c(t) \leq m_c(0)e^{-\nu t}.
\eq
Since
\bq\nonumber
\begin{aligned}
	h_{min} \leq h(t, x)+ \sigma \leq \|h(t)\|_{L^{\infty}}+\sigma, \qquad 
	(\sigma-\|h(t)\|_{L^{\infty}})^{-1} \leq (h(t, x)+ \sigma )^{-1} \leq  h_{min}^{-1}
\end{aligned}
\eq  
for any $(x, t) \in \T^d \times \mathbb{R}_{+}$,
there exists a positive constant $C= C(\|h\|_{L^\infty}, h_{min}, \gamma)$ such that
\bq\label{h}
|(h(x, t)+\sigma)^N -\sigma^N| \leq  C |h(x, t)|
\eq 
by the mean value theorem and hence,
\bq\label{L0}
\int |u|^2  + |h|^2 \, dx \leq C \mathcal{L} +(C+\frac{1}{2})|m_c|^2,
\eq 
where a constant $C= C(\|h\|_{\infty}, h_{min}, \gamma)$ may be different from the constant $C$ in \eqref{h}.
From \eqref{L0} together with \eqref{exp_mc},
it is enough to show the exponential decay rate for $\mathcal{L}(t)$ in the proposition below. 
\begin{proposition}\label{estimate_Lt}
		Let $(h, u)$ be a global classical solution to \eqref{np_ER2} with sufficient integrability.
		Suppose that
		\begin{itemize}
			\item[(i)] $\inf_{(x, t) \in \T^d \times \mathbb{R}_{+}} h(x, t) +\sigma \ge h_{min} >0$ and 
			\item[(ii)] $(h,u)\in L^{ \infty}(\mathbb{R}_{+}; L^{\infty}(\T^d))\times L^{ \infty}(\mathbb{R}_{+}; (L^{\infty}(\T^d))^d)$.
		\end{itemize}
		Then, for sufficiently small $\lambda > 0$, we have 
		\bq\nonumber
		\mathcal{L}(t) \leq C\mathcal{L}(0) e^{-\mu t}
		\eq
		for $t \geq 0$, where $C>0$ is independent of $t$.
		\end{proposition}

For the proof of Proposition \ref{estimate_Lt}. we start with the following auxiliary lemma whose proof can be found in \cite{C}. 

\begin{lemma}\label{square form}
	Let $r_0, \tilde{r}>0$ and $\gamma\ge1$ be given constants, and set
	\bq\label{f}
	f(\gamma, r; r_0) = r \int_{r_0}^r \frac{s^{\gamma}-r_0^{\gamma}}{s^2}\, ds
	\eq
	for $r\in [0, \tilde{r}]$.
	Then there exists a constant $C\ge 1$ such that
	\bq\nonumber
	\frac{1}{C(\gamma, r_0, \tilde{r})} (r-r_0)^2 \leq f(\gamma, r; r_0) \leq C(\gamma, r_0, \tilde{r})(r-r_0)^2
	\eq
	for all $r\in [0, \tilde{r}]$.
\end{lemma}

Then we consider the following modulated energy function $\mathcal{E}$:
\bq\nonumber
\begin{aligned}
	\mathcal{E} = & \int_{\T^d} (h+\sigma)^N |u-m_c|^2\, dx 
	+\frac{2}{N(N+2)} \int f\lt(\, \frac{N+2}{N}, (h+\sigma)^N; \sigma^N \, \rt) dx \\
	& -\frac{\lambda }{\sigma^{N}} \int | \Lambda^{\frac{\alpha-d}{2}} \{ (h+\sigma)^N -\sigma^N\}|^2\, dx+ \frac{1}{2}|m_c(t)|^2,
\end{aligned}
\eq 
where $f$ is given in \eqref{f}.
Then, it follows from Lemma \ref{square form} that 
\bq\nonumber
 f \lt(\frac{N+2}{N}, (h(x)+\sigma)^N; \sigma^N \rt) \sim |(h+\sigma)^N -\sigma^N|^2,
\eq  
and moreover, it is not difficult to check that 
\bq\nonumber
\int|(h+\sigma)^N -\sigma^N|^2\, dx \sim \int |(h+\sigma)^N -\sigma^N|^2-\lambda|\Lambda^{\frac{\alpha-d}{2}}\{(h+\sigma)^N -\sigma^N\}|^2\, dx 
\eq
by Lemma \ref{Sobolev_T} with small $0<\lambda <1$. Thus we attain $\mathcal{E}(t)\sim\mathcal{L}(t)$.

\begin{lemma}\label{lem_modenergy}
Let $(h, u)$ be a global classical solution to \eqref{np_ER2} with sufficient integrability. Then we have
	\bq\nonumber
	\frac{d}{dt} \mathcal{E}(t) +\mathcal{D}(t) =0,
	\eq
	where 
	\bq \nonumber
	\mathcal{D}(t)=2\nu \int (h+\sigma)^N |u-m_c|^2\, dx + \nu |m_c|^2.
	\eq
\end{lemma}

\begin{proof}
		From the system \eqref{np_ER2} with $\gamma =(N+2)/N$, we can easily find 
	\bq\nonumber
	\begin{aligned}
		&\frac{1}{2}  \frac{d}{dt} \int (h+\sigma)^N |u-m_c|^2\, dx 	+ \nu \int (h+\sigma)^N |u-m_c|^2 \, dx\\
		&\quad=- \frac{1}{ N(N+2) } \int u \cdot \nabla(h+\sigma)^{N+2}\, dx \\
		&\qquad + \nu  \int (h+\sigma)^N m_c \cdot (u-m_c)\, dx + \frac{\lambda}{\sigma^N} \int (h+\sigma)^N (u-m_c) \cdot \nabla \Lambda^{\alpha-d}\{(h+\sigma)^N -\sigma^N\}\, dx\\
		&\quad=:\sum_{i=1}^3 \sfI_{i}.
	\end{aligned}
	\eq
	For $\sfI_1$ and $\sfI_2$, we obtain
	\bq\nonumber
	\sfI_1= -\frac{1}{N(N+2)}\frac{d}{dt} \int f\lt(\frac{N+2}{N}, (h+\sigma)^N; \sigma^N\rt)\, dx 
	\eq	
	and 
	\bq\nonumber
	\sfI_2 = \nu \sigma^{N} |m_c|^2 -\nu |m_c|^2 \int (h+\sigma)^N \, dx =0.
	\eq 

	For $\sfI_3$, one uses the integration by parts to get
	\bq\nonumber
	\begin{aligned}
	\sfI_3 &=\int (h+\sigma)^N (u-m_c) \cdot \nabla \Lambda^{\alpha-d}\{(h+\sigma)^N -\sigma^N \} \, dx\\
	&= \frac{\lambda}{\sigma^N}\int \pa_t \{(h+\sigma)^N -\sigma^N\}\, \Lambda^{\alpha-d} \{(h+\sigma)^N-\sigma^N \}\, dx 
	\\
	&= \frac{\lambda}{\sigma^N} \int  \pa_t \lt( \Lambda^{\frac{\alpha-d}{2}}\{(h+\sigma)^N -\sigma^N\} \rt)\, \Lambda^{\frac{\alpha-d}{2}} \{(h+\sigma)^N -\sigma^N\}\, dx \\
	&=\frac{\lambda}{2 \sigma^N} \frac{d}{dt} \int  |\Lambda^{\frac{\alpha-d}{2}} \{(h+\sigma)^N-\sigma^N \} |^2 \, dx,
	\end{aligned}
	\eq
	since
	\bq\nonumber
	\begin{aligned}
	- \int \nabla \cdot( (h+\sigma)^N m_c)\, \Lambda^{\alpha-d}\{(h+\sigma)^N -\sigma^N \} \, dx  &=
	- \int \nabla \{ (h+\sigma)^N-\sigma^N \} \cdot m_c\, \Lambda^{\alpha-d}\{(h+\sigma)^N -\sigma^N \} \, dx\\
	& = \int \{ (h+\sigma)^N-\sigma^N \} m_c \cdot \nabla \Lambda^{\alpha-d}\{(h+\sigma)^N -\sigma^N \} \, dx\\
	& =0
	\end{aligned}
	\eq
by the symmetry of the operator $\Lambda^{\alpha-d}$ and the fact $\nabla m_c =0$. On the other hand,  one has
	\bq\nonumber
	 \frac{d}{dt}  |m_c(t)|^2 = 2 m_c(t) m'_c(t)  = -2 \nu |m_c(t)|^2,
	\eq 
and collect all the estimates to yield the desired result.
\end{proof}

Due to the absence of the dissipation rate for the density $h$ in the term $\mathcal{D}(t)$, we need to consider the perturbed energy functional which is eventually equivalent to the total fluctuated energy functional $\mathcal{L}(t)$. Motivated from \cite{CJ21}, we consider the perturbed energy functional: 
\bq\nonumber
\mathcal{E}^{\mu} (t) = \mathcal{E}(t) + \mu \int  (u-m_c) \cdot \nabla W*[(h+\sigma)^N -\sigma^N]\,dx
\eq 
for sufficiently small $\mu>0$ which will be determined in Lemma \ref{lower_Dmu}.
We note here that
$ \mathcal{E}(t) \sim \mathcal{E}^{\mu}(t)$. 
Then it is clear to get 
\bq\label{e_pertu}
\frac{d}{dt} \mathcal{E}^{\mu}(t) +\mathcal{D}^\mu(t) =0,
\eq
where 
\[
\mathcal{D}^{\mu}(t):= \mathcal{D}(t) - \mu \frac{d}{dt}\int  (u-m_c) \cdot \nabla W*[(h+\sigma)^N -\sigma^N]\,dx.
\]
Here the potential $W$ is given as in \eqref{eqn_W}. From \eqref{e_pertu} along with the equivalent relation $\mathcal{L}(t) \sim \mathcal{E}(t) \sim \mathcal{E}^{\mu} (t)$, 
Proposition \ref{estimate_Lt} is a direct result of the following.

\begin{lemma}\label{lower_Dmu}
	Let $(h, u)$ be any global classical solutions to the system \eqref{np_ER2}-\eqref{ini_ER2}.
	Under the same assumptions in Proposition \ref{estimate_Lt}, there exists a positive constant $C>0$ such that 
	\bq\nonumber
	\mathcal{D}^{\mu} (t) \ge C \mathcal{L}(t), \qquad t\ge0.
	\eq
\end{lemma}

\begin{proof}
Direct computation shows  
\bq\nonumber
\begin{aligned}
	&-\mu \frac{d}{dt}\int  (u-m_c) \cdot \nabla W*[(h+\sigma)^N -\sigma^N]\,dx \cr
	&\quad =-\mu \int u \otimes (u-m_c) : \nabla^2 W*[(h+\sigma)^N-\sigma^N] \, dx   +\frac{\mu}{N} \int (h+\sigma) \nabla h \cdot \nabla W*[(h+\sigma)^N-\sigma^N] \, dx \\
	&  \qquad + \mu \nu \int  (u-m_c) \cdot \nabla W*[(h+\sigma)^N-\sigma^N]\, dx   - \mu \lambda \int \nabla \Lambda^{\alpha-d}\{(h+\sigma)^N-\sigma^N\} \cdot \nabla W*[(h+\sigma)^N-\sigma^N]\, dx \\
	&  \qquad +\mu  \int   (u-m_c)\cdot \nabla W*[\nabla  \cdot ((h+\sigma)^N u ) ]\, dx\\
	&  \quad =: \sum_{i=1}^5 {\sfJ_i}(t).
\end{aligned}
\eq 
We then estimate each term ${\sfJ_i}, i=1,\dots,5$ as follows.

For $\sfJ_2$, we use the integration by parts to obtain
\bq\nonumber
\begin{aligned}
\sfJ_2 
&=\frac{\mu}{N} \int (h+\sigma) \nabla h \cdot \nabla W*[(h+\sigma)^N-\sigma^N] \, dx \\
&=\frac{\mu}{2N} \int \{(h+\sigma)^{2}-\sigma^{2}\}  \lt( (h+\sigma)^N-\sigma^N \rt) dx \\	
& \ge c_1\mu  \int \lt|(h+\sigma)^N-\sigma^N \rt|^2   dx,
\end{aligned}
\eq
where $c_1=c_1(\gamma, \|h\|_{L^\infty}, h_{min})$ is a positive constant.

We next use the lower bound assumption on $h +\sigma$ to estimate 
\[
\begin{aligned}
\sfJ_3 &= \nu \mu \int  (u-m_c) \cdot \nabla W*[(h+\sigma)^N-\sigma^N]\, dx\\
&\ge -c_2\nu \mu  \|(h+\sigma)^{\frac{N}{2}}(u-m_c)\|_{L^2} \|\nabla W* [(h+\sigma)^N-\sigma^N]\|_{L^2}\\
& \ge -c_2 \nu \mu  \|(h+\sigma)^{\frac{N}{2}}(u-m_c)\|_{L^2} \|(h+\sigma)^N-\sigma^N\|_{L^2}\\
& \ge -c_2 \nu \mu^{\frac{1}{2}} \|(h+\sigma)^{\frac{N}{2}}(u-m_c)\|_{L^2}^2 
  -c_2 \nu \mu^{\frac{3}{2}} \|(h+\sigma)^N-\sigma^N\|_{L^2},
\end{aligned}
\]
where $c_2 = c_2(h_{min}, \gamma)$ is a positive constant.

For $\sfJ_4$, similarly as the esitmate of $\sfJ_2$, we find
\[
\begin{aligned}
	\sfJ_4 
	&= -\mu \lambda \int  \Lambda^{\alpha-d}\{(h+\sigma)^N-\sigma^N\}\, [(h+\sigma)^N-\sigma^N]\, dx \\
	&\ge - \mu \lambda  \| \Lambda^{\alpha-d} (h+\sigma)^N-\sigma^N \|_{L^2} \| (h+\sigma)^N-\sigma^N \|_{L^2}\\
	& \ge -  \mu\lambda   \| (h+\sigma)^N-\sigma^N\|_{L^2}^2.  
\end{aligned}
\]

Finally, we observe
\bq\nonumber
\begin{aligned}
	\sfJ_5 &= \mu \int (u-m_c) \cdot  \nabla W* [\nabla  \cdot ((h+\sigma)^N (u-m_c))] \,dx\\
	&\quad +\mu \int  (u-m_c) \cdot  \nabla W* [\nabla  \cdot ([(h+\sigma)^N-\sigma^N] m_c)] \,dx
\end{aligned}
\eq  
and
\bq\nonumber
\begin{aligned}
	\int  (u-m_c) \cdot  \nabla W* [\nabla  \cdot ([(h+\sigma)^N-\sigma^N] m_c)]\, dx=\int m_c \otimes (u-m_c)  : \nabla^2 W*[(h+\sigma)^N-\sigma^N]\, dx
\end{aligned}.
\eq
Thus, we have
\[
\begin{aligned}
	\sfJ_1+\sfJ_5&=
	-\mu \int  (u-m_c)  \otimes (u-m_c) : \nabla^2 W*[(h+\sigma)^N-\sigma^N] \, dx \\
	& \quad + \mu \int(u-m_c) \cdot  \nabla W* [\nabla  \cdot ((h+\sigma)^N (u-m_c))] \,dx\\
	& \ge - \mu h_{min}^{-\frac{N}{2}}\| (h+\sigma)^\frac{N}{2} (u-m_c) \|_{L^2}\| (u-m_c) \|_{L^\infty} \| \nabla^2 W*[(h+\sigma)^N-\sigma^N]\|_{L^2}\\
	& \quad-\mu h_{min}^{-\frac{N}{2}} \|(h+\sigma)^{\frac{N}{2}}(u-m_c)\|_{L^2} \| \nabla W* [\nabla \cdot ((h+\sigma)^N (u-m_c))]\|_{L^{2}}\\
	&\ge -C \mu^{1/2}(1+\mu^{1/2})  \|(h+\sigma)^{\frac{N}{2}}(u-m_c)\|_{L^2}^2 - c_3\mu^{3/2}\|(h+\sigma)^N - \sigma^N\|_{L^2}^2,
\end{aligned}
\]
where we used Young's inequality and $c_3 = c_3(\gamma, m_c(0), \|u\|_{L^\infty}, h_{min})$ is a positive constant.

Hence, we collect all the estimates for $\sfJ_i$'s to deduce that 
\bq\nonumber
\begin{aligned}
	\mathcal{D}^{\mu}(t) 
\ge &\lt( 2\nu - C\mu^{\frac{1}{2}}( \nu+ \mu^{\frac{1}{2}}  ) \rt) \|(h+\sigma)^\frac{N}{2}  (u-m_c) \|_{L^2}^2 \\
&+ \mu \lt( c_1 -(c_2 \nu+c_3) \mu^{\frac{1}{2}} - \lambda \rt) \| (h+\sigma)^N-\sigma^N\|_{L^2}^2  +\nu  |m_c|^2.
\end{aligned}
\eq
We now take $\mu, \lambda>0$ sufficiently small so that we find a positive constant $C$ such that
\[
	\mathcal{D}^{\mu}(t) \ge C \mathcal{L}(t),
\]
which implies the desired result.
\end{proof}

\begin{remark}
Here, note that the condition $u \in L^\infty(\T^d \times \R_+)$ can be replaced by $ \nabla h\in L^\infty(\T^d \times \R_+)$. Indeed, we can estimate $\sfJ_1+\sfJ_5$ from above as
\[\begin{aligned}
	\sfJ_1+\sfJ_5&=
	-\mu \int  (u-m_c)  \otimes (u-m_c) : \nabla^2 W*[(h+\sigma)^N-\sigma^N] \, dx \\
	& \quad + \mu \int(u-m_c) \cdot  \nabla W* [\nabla  \cdot ((h+\sigma)^N (u-m_c))] \,dx\\
	& \ge - \mu h_{min}^{-N}\| (h+\sigma)^\frac{N}{2} (u-m_c) \|_{L^2}^2 \| \nabla^2 W*[(h+\sigma)^N-\sigma^N]\|_{L^{\infty}}\\
	& \quad-\mu h_{min}^{-\frac{N}{2}} \|(h+\sigma)^{\frac{N}{2}}(u-m_c)\|_{L^2} \| \nabla W*[ \nabla \cdot ((h+\sigma)^N (u-m_c))]\|_{L^{2}}\\
	& \ge - \mu h_{min}^{-N} 	\|\nabla W\|_{L^1}\| \nabla h \|_{L^{\infty}} \| (h+\sigma)^\frac{N}{2} (u-m_c) \|_{L^2}^2
 -\mu h_{min}^{-\frac{N}{2}}(\|h\|_{L^{\infty}}+\sigma)^{\frac{N}{2}} \|(h+\sigma)^{\frac{N}{2}}(u-m_c)\|_{L^2}^2\\
	&\ge -c_4 \mu  \|(h+\sigma)^{\frac{N}{2}}(u-m_c)\|_{L^2}^2,
\end{aligned}
\]
where $c_4 = c_4(\gamma, \|h\|_{W^{1,\infty}}, h_{min})$ is a positive constant.

\end{remark}

\section*{Acknowlegments}
The work of Y.-P. Choi and Y. Lee are supported by NRF grant no. 2022R1A2C1002820. The work of J. Jung is supported by NRF grant (No. RS-2022-00165600). The work of Y. Lee is supported by NRF grant (No. NRF-2022R1I1A1A01068481).

%
%
%
%
%
%
%
%
%
\appendix
%
%
%
%
%
%
%
%
%
\section{Proof of Lemma \ref{higher_it}}\label{app.A}
In this appendix, we provide the details of the proof of Lemma \ref{higher_it}.

Let $0<l<1/2$. For any $k\ge 0$, we obtain
\bq
\begin{aligned}
	&\frac{1}{2} \frac{d}{dt}\int (h+\sigma)^{k(N-2)} |\Lambda^{m-kl} u|^2 \, dx 
	+  \nu \int (h+\sigma)^{k(N-2)} |\Lambda^{m-kl} u|^2 \, dx
	\nonumber\\
	&\quad = \frac{k(N-2)}{2}\int (h+\sigma)^{k(N-2)-1} \pa_t h\, |\Lambda^{m-kl} u |^2\, dx  
	-\int (h+\sigma)^{k(N-2)} \Lambda^{m-kl} u \cdot \Lambda^{m-kl} (u\cdot \nabla u)\, dx 	\\		
	& \qquad-\frac{1}{N}\int (h+\sigma)^{k(N-2)}\Lambda^{m-kl} u \cdot \Lambda^{m-kl} \{ (h+\sigma) \nabla h\}\, dx	\\
	&\qquad + \lambda \sigma^{-N} \int (h+\sigma)^{k(N-2)} \Lambda^{m-kl}u\cdot \Lambda^{m-(k+2)l}\nabla \{(h+\sigma)^N -\sigma^N\} \, dx\\
	&\quad =: \sfI_1 + \sfI_2   - \mathcal{P}_{m}(k) +\mathcal{R}_m(k),
\end{aligned}
\eq 
where $\sfI_1$ can be estimated as 
\bq\nonumber
\begin{aligned}
	\sfI_1  
	&\leq C \|\pa_t h\|_{L^{\infty}}  \|\Lambda^{m-kl} u\|_{L^2}^2 \\
	&\leq C \|u\|_{L^{\infty}} \| \nabla h\|_{L^{\infty}} \|\Lambda^{m-kl} u\|_{L^2}^2
	+ C \|\nabla u \|_{L^{\infty}} \|\Lambda^{m-kl} u\|_{L^2}^2\\
	&\leq C \|u\|_{L^{\infty}} \| \nabla h\|_{L^{\infty}} \|\Lambda^{m} u\|_{L^2}^2
	+ C \|\nabla u \|_{L^{\infty}} \|\Lambda^{m} u\|_{L^2}^2
\end{aligned}
\eq 
for some $C=C(\gamma, k) > 0$.

For $\sfI_2$, we rewrite it as
\bq\label{J2}
\sfI_2 =
-\int (h+\sigma)^{k(N-2)} \Lambda^{m-kl} u \cdot ([\Lambda^{m-kl}, u\cdot] \, \nabla u)\, dx 
-\frac{1}{2}\int (h+\sigma)^{k(N-2)}  u\cdot \nabla (|\Lambda^{m-kl} u|^2)\, dx. 
\eq
We then use Lemma \ref{lem_moser} to yield
\[
\begin{aligned}
	-\int (h+\sigma)^{k(N-2)} \Lambda^{m-kl} u \cdot ([\Lambda^{m-kl}, u\cdot] \, \nabla u)\, dx 
	&\leq C \| \Lambda^{m-kl} u \|_{L^2} \| [\Lambda^{m-kl}, u\cdot] \,\nabla u\|_{L^2}\\ 
	& \leq C  \| \Lambda^{m-kl} u \|_{L^2} \|u\|_{H^{\frac{d}{2}+1+\varepsilon}}( \| \Lambda^{m-kl} u \|_{L^2} +\| \Lambda^m u \|_{L^2})\\
	& \leq C \|u\|_{H^m} \| \Lambda^{m} u \|_{L^2}^2,  
\end{aligned}
\]
where $\e>0$ is a sufficiently small constant and $kl<m$. For the second term on the right-hand side of \eqref{J2}, we integrate by part to get
\[
\begin{aligned}
	&-\frac{1}{2}\int (h+\sigma)^{k(N-2)} u\cdot \nabla (|\Lambda^{m-kl} u|^2 )\, dx \\
	&\quad =\frac{k(N-2)}{2}\int (h+\sigma)^{k(N-2)-1} \nabla h \cdot u\, |\Lambda^{m-kl} u|^2\, dx +\frac{1}{2}\int (h+\sigma)^{k(N-2)}  (\nabla \cdot u) |\Lambda^{m-kl} u|^2\, dx\\
	&\quad \leq C\|\nabla h\|_{L^{\infty}} \|u\|_{L^{\infty}} \|\Lambda^{m-kl} u\|_{L^2}^2 
	+C \|\nabla \cdot u\|_{L^{\infty}} \|\Lambda^{m-kl} u \|_{L^2}^2\\
	&\quad \leq C\|\nabla h\|_{L^{\infty}} \|u\|_{L^{\infty}}\|\Lambda^{m} u\|_{L^2}^2 
	+C \|\nabla \cdot u\|_{L^{\infty}} \|\Lambda^{m} u \|_{L^2}^2	
\end{aligned}
\]
for some $C=C(\gamma, k) > 0$. Thus, we deduce
\[
\sfI_2 
\leq  C\|u\|_{H^m}\|\Lambda^{m} u \|_{L^2}^2 + C \| h\|_{H^m} \|u\|_{H^m} \|\Lambda^{m} u\|_{L^2}^2.
\]
Combining the above estimates gives
\bq\nonumber
\begin{aligned}
	&\frac{1}{2} \frac{d}{dt}\int (h+\sigma)^{k(N-2)} |\Lambda^{m-kl} u|^2 \, dx  +\nu \int (h+\sigma)^{k(N-2)} |\Lambda^{m-kl} u|^2 \, dx \\
	&\quad \leq C\|u\|_{H^m}\|\Lambda^{m} u \|_{L^2}^2 +C \| h\|_{H^m} \|u\|_{H^m} \|\Lambda^{m} u\|_{L^2}^2 -\mathcal{P}_{m}(k) +\mathcal{R}_m (k), 
\end{aligned}
\eq 
where $C = C(\gamma, m, k, l) > 0$. This together with using the fact $h+\sigma>0$ concludes the desired estimate \eqref{estimate_Pk}.

%
%
%
%
%
%
%
%
%

%
%
%
%

\end{document}